\documentclass[11pt,a4paper]{article} 
\usepackage[margin=1in]{geometry}
\usepackage{graphicx}
\usepackage{enumerate}
\usepackage{color}
\usepackage{url}
\usepackage{hyperref}
\usepackage[french,main=UKenglish]{babel}
\usepackage[headings]{fullpage} 
\usepackage{amsmath,amsthm}
\usepackage{amsfonts}
\usepackage{amssymb}
\usepackage{mathrsfs}

\usepackage{latexsym}
\usepackage{array} 
\usepackage{mathrsfs}
\usepackage{verbatim}
\usepackage{amsbsy}
\usepackage[toc,page]{appendix}

\newcommand{\we}{\wedge}
\newcommand{\ol}{\overline}

\newcommand{\eps}{\varepsilon}

\newcommand{\ba}{\begin{array}}
	\newcommand{\ea}{\end{array}}
\newcommand{\be}{\begin{equation}}
\newcommand{\ee}{\end{equation}}
\newcommand{\bea}{\begin{eqnarray}}
\newcommand{\eea}{\end{eqnarray}}
\newcommand{\beaa}{\begin{eqnarray*}}
	\newcommand{\eeaa}{\end{eqnarray*}}

\def\dbE{\mathbb{E}}

\def\dbL{\mathbb{L}}

\def\dbN{\mathbb{N}}
\def\dbP{\mathbb{P}}
\def\dbR{\mathbb{R}}

\def\a{\alpha}
\def\b{\beta}
\def\g{\gamma}
\def\d{\delta}
\def\e{\varepsilon}
\def\z{\zeta}

\def\l{\lambda}
\def\m{\mu}

\def\si{\sigma}
\def\t{\tau}
\def\f{\varphi}

\def\D{\Delta}

\def\cC{{\cal C}}

\def\cE{{\cal E}}
\def\cF{{\cal F}}

\def\cK{{\cal K}}

\def\cP{{\cal P}}

\def\cS{{\cal S}}

\def\cW{{\cal W}}

\def\no{}

\def\ms{\medskip}

\def\q{\quad}

\def\pa{\partial}
\def\cd{\cdot}

\newcommand{\basa}{\begin{assumption}}
	\newcommand{\easa}{\end{assumption}}

\newcommand{\bas}{\begin{assum}}
	\newcommand{\eas}{\end{assum}}

\def\limsup{\mathop{\overline{\rm lim}}}
\def\liminf{\mathop{\underline{\rm lim}}}

\def\pa{\partial}

 \def\cd{\cdot}

\def\1{{\bf 1}}

\def\:{\!:\!}

\newtheorem{thm}{Theorem}[section]
\newtheorem{lem}[thm]{Lemma}
\newtheorem{cor}[thm]{Corollary}
\newtheorem{prop}[thm]{Proposition}
\newtheorem{rem}[thm]{Remark}
\newtheorem{eg}[thm]{Example}

\newtheorem{assum}[thm]{Assumption}

\DeclareMathOperator*{\argmin}{\arg\!\min}

\DeclareMathOperator{\Tr}{Tr}

\numberwithin{equation}{section}
\numberwithin{thm}{section}

\begin{document}

\title{Mean-Field Langevin Dynamics and Energy Landscape of Neural Networks}
\author{Kaitong HU \footnote{CMAP, \'Ecole Polytechnique, F-91128 Palaiseau Cedex, France.} \and
	\and Zhenjie REN \footnote{CEREMADE, Universit\'e Paris Dauphine, F-75775 Paris Cedex 16, France. } \and
	David \v{S}I\v{S}KA \footnote{School of Mathematics, University of Edinburgh, James Clerk Maxwell Building, Peter Guthrie Tait Road, 
		Edinburgh EH9 3FD, UK.} \and
	{\L}ukasz SZPRUCH \footnote{School of Mathematics, University of Edinburgh, James Clerk Maxwell Building, Peter Guthrie Tait Road, 
		Edinburgh EH9 3FD, UK.}
}
\maketitle

\def\L{\Lambda}
\begin{abstract}
	Our work is motivated by a desire to study the  theoretical underpinning for the convergence of stochastic gradient type algorithms widely used for non-convex learning tasks such as training of neural networks. 
	The key insight, already observed in \cite{mei2018mean,
		chizat2018global,rotskoff2018neural}, is that a certain class of the finite-dimensional non-convex problems becomes convex when lifted to infinite-dimensional space of measures.  We leverage this observation and show that the corresponding energy functional defined on the space of probability measures has a unique minimiser which can be characterised by a first-order condition using the notion of linear functional derivative. 
	Next, we study the corresponding gradient flow structure in 2-Wasserstein metric, which we call Mean-Field Langevin Dynamics (MFLD), and show that the flow of marginal laws induced by the  gradient flow converges to a stationary distribution, which is exactly the minimiser of the energy functional. We observe that this convergence is exponential under conditions that are satisfied for highly regularised learning tasks. 
	Our proof of convergence to stationary probability measure is novel and it relies on a generalisation of LaSalle's invariance principle combined with HWI inequality. 
	Importantly, we assume neither that interaction potential of MFLD is of convolution type nor that it has any particular symmetric structure. 
	Furthermore, we allow for the general convex objective function, unlike, most papers in the literature that focus on quadratic loss.  
	Finally, we show that the error between finite-dimensional optimisation problem and its infinite-dimensional limit is of order one over the number of parameters.
\end{abstract}

\begin{otherlanguage}{french}
\begin{abstract}
	L'objectif de nos travaux est d'\'etudier le fondement th\'eorique pour la convergence des algorithmes du type gradient stochastique, qui sont tr\`es souvent utilis\'es dans les probl\`emes d'apprentissage non-convexe, e.g. calibrer un réseau de neurones. L'observation cl\'e, qui a d\'ej\`a \'et\'e remarqu\'ee dans \cite{mei2018mean,chizat2018global,rotskoff2018neural}, est qu'une certaine classe de probl\`emes non-convexes fini-dimensionnels devient convexe une fois inject\'ee dans l'espace des mesures de probabilit\'e. \`A l'aide de cette observation nous montrons que la fonction d'\'energie correspondante d\'efinie dans l'espace des mesures de probabilit\'e a un unique minimiser qui peut \^etre caract\'eris\'e par une condition de premier ordre en utilisant la notion de d\'eriv\'ee fonctionnelle. Par la suite, nous \'etudions la structure de flux de gradient avec la m\'etrique de 2-Wasserstein, que nous appelons la dynamique de Langevin au champs moyen (MFLD), et nous montrons que la loi marginale du flux de gradient converge vers une loi stationnaire  qui correspond au minimiser de la m\^eme fonction d'\'energie pr\'ec\'edente. Sous certaines conditions de r\'egularit\'e du probl\'eme initial, la convergence a lieu \`a une vitesse exponentielle . Nos preuves de la convergence vers la loi stationnaire est nouvelle, qui reposent sur le principe d'invariance de LaSalle et l'in\'egalit\'e HWI. Remarquons que nous ne supposons pas que l'interaction potentielle de MFLD soit du type convolution ou sym\'etrique. De plus, nos r\'esultats s'appliquent aux fonctions d'objectif convexes g\'en\'erales contrairement aux beaucoup de papiers dans la litt\'erature qui se limitent aux fonctions quadratiques. Enfin, nous montrons que la diff\'erence entre le probl\'eme initial d'optimisation fini-dimensionnel et sa limite dans l'espace des mesures de probabilit\'e est de l'ordre d'un sur le nombre de param\`etres.  
\end{abstract}
\end{otherlanguage}


	
\section{Introduction}

Neural networks trained with stochastic gradient descent algorithm proved to be extremely successful in number of applications such as computer vision, natural language processing, generative models or reinforcement learning~\cite{lecun2015deep}. 
However, complete mathematical theory that would provide theoretical guarantees for the convergence 
of machine learning algorithms for non-convex learning tasks has been elusive. 
On the contrary, empirical experiments demonstrate that classical learning theory \cite{vapnik2013nature} may fail to predict the behaviour of modern machine learning algorithms \cite{zhang2016understanding}. In fact, it has been observed that the performance of neural networks based algorithms is insensitive to the number of parameters in the hidden layers (provided that this is sufficiently large)  and in practice one works with models that have number of parameters larger than the size of the training set \cite{hastie2019surprises,belkin2018does}. These findings motivate the study of neural networks with large number of parameters which is a subject of this work.  

Furthermore while universal representation theorems ensures the existence of the optimal parameters of the network, it is in general not known when such optimal parameters can be efficiently approximated by conventional algorithms, such as stochastic gradient descent. This paper aims at revealing the intrinsic connection between the optimality of the network parameters and the dynamic of gradient-descent-type algorithm, using the perspective of the mean-field Langevin equation. 

This work builds on the rigorous mathematical framework to study non-convex learning tasks such as training of neural networks developed
in Mei, Misiakiewicz and Montanari~\cite{mei2018mean}, Chizat and Bach~\cite{chizat2018global}, Sirignano and Spiliopoulos ~\cite{sirignano2018mean} 
as well as Rotskoff and Vanden-Eijnden~\cite{rotskoff2018neural}.

We extend some existing results and provide a novel proof technique for mathematical results which provide a theoretical underpinning for the convergence of stochastic gradient type algorithms widely used in practice to train neural networks. We demonstrate how our results apply to a situation when one aims to  train one-hidden layer neural network with (noisy) stochastic gradient algorithm. 

Let us first briefly recall the classical finite dimensional Langevin equation. Given a {\it potential} function $f:\dbR^d\rightarrow\dbR$ which is Lipschitz continuous and satisfies appropriate growth condition, the overdamped Langevin equation reads
\bea\label{classicLangevin}
\mathrm{d}X_t = -\nabla f(X_t) \mathrm{d}t + \si \mathrm{d}W_t,
\eea 
where $\si$ is a scalar constant and $W$ is a $d$-dimension Brownian motion. One can view this dynamic in two perspectives:
\begin{enumerate}[i)]
	\item The solution to \eqref{classicLangevin} is a time-homogeneous Markov diffusion, so under mild condition it admits a unique invariant measure $m^{\si,*}$, of which the density function must be in the form
	\beaa
	m^{\si,*}(x) = \frac1Z \exp\left(-\frac{2}{\si^2} f(x)\right), \q\mbox{for all}\q x\in \dbR^d\,,\,\,\,\text{where}\,\, Z:= \int_{\mathbb R^d} \exp\left(-\frac{2}{\si^2} f(x)\right)  \,\mathrm{d}x\,.
	\eeaa
	
	\item The dynamic  \eqref{classicLangevin} can be viewed as the  path of a  randomised continuous time gradient descent algorithm. 
\end{enumerate}
These two perspectives are unified through the variational form of the invariant measure, namely, $m^{\si,*}$ is the unique minimiser of the free energy function 
\[
V^\si(m) := \int_{\mathbb R^d} f(x) m(\mathrm{d}x) + \frac{\si^2}{2} H(m)
\]
over all probability measure $m$, where $H$ is the relative entropy with respect to the Lebesgue measure. The variational perspective has been established in \cite{jordan199618} and \cite{JKO98}. Moreover, one may observe that the distribution $m^{\si,*}$ concentrates to the Dirac  measure $\d_{\arg\min f}$  as $\si\rightarrow 0$ and there is no need to  assume that the function $f$ is convex. This establishes the link between theory of statistical sampling and optimisation and show that Langevin equation plays an important role in the non-convex optimisation. This fact is well-recognized by the communities of numerical optimisation and machine learning \cite{hwang1980laplace,holley1989asymptotics,holley1988simulated}

This paper aims at generalising the connection between the global minimiser and the invariant measure to the case where the {\it potential} function is a function defined on a space of probability measures.
This is motivated by the following observation on the configuration of neural network. Let us take the example of the network with $1$-hidden-layer. While the universal representation theorem, \cite{cybenko1989approximation,barron1993universal} tells us that $1$-hidden-layer network can arbitrarily well approximate the continuous function on the compact time interval it does not tell us how to find optimal parameters. One is faced with the following non-convex optimisation problem. 
\begin{equation}
	\label{intro:min0}
	\min_{\beta_{n,i}\in\dbR, \a_{n,i}\in \dbR^{d-1}} \left\{  \int_{\dbR\times\dbR^{d-1}} \Phi\Big( y -  \frac1n\sum_{i=1}^n \beta_{n,i} \f(\a_{n,i}\cd z) \Big) \nu(\mathrm{d}y,\mathrm{d}z) \right\},	
\end{equation}
where $\Phi:\dbR\rightarrow\dbR$ is a convex function,  $\f:\dbR\rightarrow\dbR$ is a bounded, continuous, non-constant activation function and $\nu$ is a measure of compact support representing the data. 
Let us define the empirical law of the parameters as
$m^n:= \frac1n\sum_{i=1}^{n} \delta_{\{ \beta_{n,i},\a_{n,i} \}}$. 
Then
\[
\frac1n\sum_{i=1}^n \beta_{n,i} \f(\a_{n,i}\cd z) = \int_{\mathbb R^d} \beta \f(\a\cd z) \,m^n(\mathrm{d}\beta,\mathrm{d}\alpha)\,.
\]
To ease notation let us use, for $x=(\beta,\alpha)\in \mathbb R^d$, the function $\hat \varphi(x,z) := \beta\varphi(\alpha \cdot z)$, and by $\dbE^m$ we denote the expectation of random variable $X$ under the probability measure $m$.
Now, instead of~\eqref{intro:min0}, we propose to study the following minimisation problem over the probability measures:
\bea\label{intro:min}
\min_{m} F(m),\q\mbox{with}\q F(m):= \int_{\dbR^d} \Phi\Big( y -  \dbE^m\big[ \hat \varphi(X,z)\big] \Big) \, \nu(\mathrm{d}y,\mathrm{d}z),
\eea 
This reformulation is crucial, because the {\it potential} function $F$ defined above is convex in the measure space
i.e. for any probability measures $m$ and $m'$ it holds that
\[
F((1-\alpha)m + \alpha m') \leq (1-\alpha)F(m) + \alpha F(m') \,\,\, \text{for all}\,\,\, \alpha \in [0,1]\,.
\]
This example demonstrates that a non-convex minimisation problem on a finite-dimensional parameter space becomes a convex minimisation problem when lifted
to the infinite dimensional space of probability measures. 
The key aim of this work is to provide analysis that takes advantage of this observation. 

In order to build up the connection between the global minimiser of the convex potential function $F$ and the upcoming mean-field Langevin equation, as in the classic case, we add the relative entropy $H$ as a regulariser, but different from the classic case, we use the relative entropy with respect to a Gibbs measure of which the density is proportional to $e^{-U(x)}$. A typical choice of the Gibbs measure could be the standard Gaussian distribution. 
One of our main contributions is to characterise the minimiser of the free energy function 
\[
V^{\sigma}:= F + \frac{\si^2}{2} H
\] 
using the  {\it linear functional derivative} on the space of probability measures,
denoted by $\frac{\d }{\d m}$ 
(introduced originally in calculus of variations
and now used extensively in the theory of mean field games see, e.g. Cardaliaguet et al.~\cite{cardaliaguet2015master}). 
Indeed, we prove the following first order condition:
\beaa
m^* =\arg\min_m V^{\sigma}(m) \q 
\text{if and only if} \q 
\frac{\d F}{\d m}(m^* , \cd)  + \frac{\si^2}{2}  \log(m^{\ast} )+ \frac{\sigma^2}{2} \, U = \mbox{\it constant}.
\eeaa
This condition together with the fact that $m^{\ast}$ is a probability measure gives 
\[
m^{\ast}(x) = \frac{1}{Z} \exp{\left( - \frac{2}{\sigma^2 }\left(\frac{\d F}{\d m}(m^* , x) + U(x) \right)  \right)}\,,
\]
where $Z$ is the normalising constant. We emphasise that throughout $V$ and hence $m^\ast$ depend on the regularisation parameter $\sigma>0$.
It is noteworthy that the variational form of the invariant measure of the classic Langevin equation is  a particular example of this first order condition. Moreover, given a measure $m^{*}$ satisfying the first order condition, it is formally a stationary solution to the nonlinear Fokker--Planck equation:
\begin{equation} \label{eq pdeintro}
	\pa_t m = \nabla\cd \bigg( \Big(D_m F(m,\cd) + \frac{\sigma^2}{2} \, \nabla U\Big) m +  \frac{\si^2}{2} \nabla m\bigg),
\end{equation}
where $D_m F$ is the {\it intrinsic derivative} on the probability measure space, defined as $D_mF(m,x):= \nabla \frac{\d F}{\d m}(m,x)$. Clearly, the particle dynamic corresponding to this Fokker-Planck equation is governed by the {\it mean field Langevin equation}:
\bea\label{eq:langeintro}
\mathrm{d}X_t = -  \Big(D_m F(m_t , X_t ) + \frac{\sigma^2}{2} \, \nabla U(X_t) \Big) \mathrm{d}t +\si \mathrm{d}W_t, \q\mbox{where}~~ m_t:= \mbox{\rm Law}(X_t).
\eea
Therefore, formally, we have already obtained the correspondence between the minimiser of the free energy function and the invariant measure of \eqref{eq:langeintro}. In this paper, the connection is rigorously proved mainly with a probabilistic argument. 

For the particular application to the neural network \eqref{intro:min}, it is crucial to observe that the dynamics corresponding to the mean field Langevin dynamics describes exactly the path of the randomised regularized gradient-descent algorithm. More precisely, consider the case where 
we are given data points $(y_m, z_m)_{m\in \mathbb N}$ which are i.i.d. samples from $\nu$. If the loss function $\Phi$ is simply the square loss then a version of the (randomized, regularized) gradient descent algorithm for the evolution of parameter $x_k^i$ will simply read as
\begin{equation}\label{eq sgd}
	x^i_{k+1} = x^i_k + 2\tau\left( \bigg(y_k - \frac1N \sum_{j=1}^N\hat \varphi(x_k^j, z_k) \bigg) \nabla \hat \varphi(x_k^i, z^k)  - \frac{\sigma^2}{2}  \nabla U(x^i_k) \right) + \si \sqrt{\tau}\xi_k^i\,,
\end{equation}
with $\xi_k^i$ independent samples from $N(0,I_d)$ (for details we refer the reader to Section~\ref{sec gradient descent}). 
This evolution is an approximation of~\eqref{eq:langeintro} and  can be viewed as noisy gradient decent. 
Indeed, in its original form, the classical stochastic gradient decent (also known as the Robins--Monroe algorithm), is given by \eqref{eq sgd} with $\sigma=0$. 

\subsection{Organisation of the Paper}	
The introduction is concluded by Section~\ref{sec lit rev}, where we compare the findings in this paper to those available in the literature, and by Section~\ref{sec meas der} recall some basic notions of measure derivatives. 
All the main results of the paper are presented in Section~\ref{sec main}.
In Section~\ref{sec applications and sgd} we show how the results in Section~\ref{sec main} apply to in the case of gradient descent training of neural networks. 
Section~\ref{sec free energy fn} contains all the proofs of the results concerning the free energy function: $\Gamma$-convergence when $\sigma \to 0$, particle approximation and the first order condition.
In Section~\ref{sec mf lang} we prove required properties of~\eqref{eq pdeintro} and~\eqref{eq:langeintro}, Section~\ref{sec conv} is used to prove the convergence of the  solution to~\eqref{eq pdeintro} to an invariant measure which is the minimizer of the free energy function.   

\subsection{Theoretical Contributions and Literature Review}
\label{sec lit rev}

The study of stationary solutions to nonlocal, diffusive equations \eqref{eq pdeintro} is classical topic with it roots in statistical physics literature and with strong links to Kac's program in Kinetic theory \cite{mischler2013kac}. We also refer reader to excellent monographs \cite{bakry1985diffusions} and \cite{ambrosio2008gradient}. In particular, variational approach has been developed in \cite{carrillo2003kinetic,otto2001geometry,tugaut2013convergence} where authors studied dissipation of entropy for granular media equations with the symmetric interaction potential of convolution type (interaction potential corresponds to term $D_mF$ in \eqref{eq pdeintro}). We also refer a reader to similar results with proofs based on particle approximation of \cite{cattiaux2008probabilistic,veretennikov2006ergodic,bolley2013uniform}, coupling arguments \cite{eberle2019quantitative} and Khasminskii's technique \cite{butkovsky2014ergodic,bogachev2019convergence}. All of the above results  impose restrictive condition on interaction potential or/and require it to be sufficiently small. We manage to relax these assumptions allowing for the interaction potential to be arbitrary (but sufficiently regular/bounded) function of measure. Our proof is probabilisitic in nature. Using Lasalle's invariance principle and the HWI inequality from Otto and Villani \cite{OV00} as the main ingredients, we prove the desired convergence. This approach, to our knowledge, is original, and it clearly justifies the solvability of the randomized/regularized gradient descent algorithm for neural networks.  
Finally we clarify how different notions of calculus on the space of probability measures enter our framework. The calculus is critical to work with arbitrary functions of measure. We refer to \cite[Chapter 5]{Carmona+Rene} for an overview on that topic. The calculus on the measure space enables to derive and quantify the error 
between finite dimensional optimisation problem and its infinite dimensional limit.

Other results are now available for the mean-field description of non-convex learning problems, see~\cite{mei2018mean,mei2019mean,chizat2018global,rotskoff2018neural,javanmard2019analysis,sirignano2018mean}. 
Let us compare this paper to the key results available in the literature. 
There are essentially three, or, if entropic regularization is included, four key ingredients: 
\begin{enumerate}[i)]
	\item \label{itm chaos} that the finite dimensional optimisation problem is approximated by infinite dimensional problem of minimizing over measures (Theorem~\ref{th static particles}),
	\item \label{itm ent} that the regularized problem approximates the original minimization problem (Proposition~\ref{prop:Gamma}),
	\item \label{itm foc} that on the space of probability measures the minimizers (or, if entropic regularization is included, the unique minimizer) satisfy a first order condition (Proposition~\ref{prop:firstorder}),
	\item \label{itm conv} and finally that on the space of probability measures we have a gradient flow that converges with time to the minimizer (Theorem~\ref{thm:convergence}). 
\end{enumerate}

Chizat and Bach~\cite{chizat2018global} work without adding entropic regularization which means that their minimization task is convex but not strictly convex. 
They have results regarding~\ref{itm chaos}) and \ref{itm foc}). 
They have a partial result related to~\ref{itm conv}) in that they prove that if the gradient flow converges to a limit, as $t\to \infty$ then objective function also converge to global minimiser.
To prove this final convergence result they require the assumption that the activation function is homogenous of either order 1 or order 2 and that, essentially, initial law has full support. The setting used in Chizat and Bach~\cite{chizat2018global} is rather different to that of the results in the present paper: in particular we do not need to assume any homogeneity on the activation function and apart mild integrability conditions we do not make assumptions on the initial law. Since we regularize using entropy we obtain convergence of the gradient flow to the minimizer. 

Rotskoff and Vanden-Eijnden~\cite{rotskoff2018neural} again work without entropic regularization and have results~\ref{itm chaos}). 
Moreover they provide Central-Limit-Theorem-type fluctuation results. 
They do show that the output of the network converges to a limit as the time in the gradient flow for the parameter measure goes to infinity. 
However they do not prove convergence of the parameter measure itself as in~\ref{itm conv}).

Sirignano and Spiliopoulos in \cite{sirignano2018mean} provide detail analysis of \ref{itm chaos}), also studying time-discretisation of continuous time gradient flow. In Section~\ref{sec gradient descent}, by using links between intrinsic derivative on the space of measure and its finite dimensional projection we provide further insight on the choice of scalings needed to derive non-trival limit in \cite{sirignano2018mean}.

The setting of our paper is the closest to that of Mei, Misiakiewicz and Montanari~\cite{mei2018mean} in that they use the entropic regularization. 
For a square loss function $\Phi$ and a quadratic regularizer $U$ they prove results on all of~\ref{itm chaos}), \ref{itm ent}), \ref{itm foc}) and~\ref{itm conv}). 
In this paper we allow a general loss function for all the above results(that this is possible is conjectured in Appendix B of \cite{mei2018mean}). 
Due to the special choice of the square loss function, in \cite[Lemma 6.10]{mei2018mean} the authors can compute directly the dynamics of $F(m_t)$ along the flow of measures defined  by \eqref{eq pdeintro}. Instead, we obtain the desired dynamics for much more general $F$ using a pathwise argument based on the It\^o calculus (see Theorem \ref{thm:Vdecrease}). The proof of~\ref{itm conv}) in~\cite[Lemma 6.12]{mei2018mean} is based on the Poincar\'e inequality for the Gaussian distribution and shows that the marginal law weakly  converges. It is not clear whether this argument can be extended to the case with a general regularizer $U$,
whereas this paper develops a new technique based on LaSalle's invariance principle and the HWI inequality, 
which allows us to prove the convergence for general $U$ in the Wasserstein-2 metric, and moreover we observe that in the highly regularized case this convergence is exponential, see Theorem \ref{thm:convergence}.  

\subsection{Calculus on the Space of Probability Measures}
\label{sec meas der}

By $\cP(\dbR^d)$ we denote the space of probability measures on $\dbR^d$, and by $\cP_p(\dbR^d)$ the subspace of $\cP(\dbR^d)$ in which the measures have finite $p$-moment for $p\ge 1$. 
Note that $\pi \in \mathcal P_p(\mathbb R^d \times \mathbb R^d)$ is called a coupling of $\mu$ and $\nu$ in $\mathcal P_p(\mathbb R^d)$, if for any borel subset $B$ of $\mathbb R^d$ we have $\pi(B,\mathbb R^d) = \mu(B)$ and $\pi(\mathbb R^d,B) = \nu(B)$.
By $\cW_p$ we denote the Wasserstein-$p$ metric on $\cP_p(\dbR^d)$, namely,
\beaa
\cW_p (\mu,\nu):= \inf\left\{ \Big(\int_{\dbR^d\times\dbR^d} |x-y|^p \pi(\mathrm{d}x,\mathrm{d}y)\Big)^{\frac1p} ;~ \mbox{$\pi$ is a coupling of $\mu$ and $\nu$}\right\}\q\mbox{for}\q \mu,\nu\in \cP_p(\dbR^d).
\eeaa
It is convenient to recall that
\begin{enumerate}[i)]
	\item $\big(\cP_p(\dbR^d),  \cW_p\big)$ is a Polish space;
	\item $\cW_p(\mu_n, \mu)\rightarrow 0$ if and only if $\m_n$ weakly converge to $\mu$ and $\int_{\dbR^d} |x|^p \mu^n(\mathrm{d}x)\rightarrow \int_{\dbR^d} |x|^p \mu(\mathrm{d}x)$;
	\item for $p'>p$, the set $\{\mu\in \cP_p(\dbR^d):  \int_{\dbR^d} |x|^{p'} \mu(\mathrm{d}x) \le C\}$ is $\cW_p$-compact.
\end{enumerate}

We say a function $F:\cP(\dbR^d)\rightarrow \dbR$ is in $\cC^1$ if there exists a bounded continuous function $\frac{\d F}{\d m}: \cP(\dbR^d)\times \dbR^d \rightarrow \dbR$ such that
\bea\label{C1derv}
F(m') - F(m) = \int_0^1 \int_{\dbR^d} \frac{\d F}{\d m}\big((1-\l)m+\l m', x\big)(m'-m)(\mathrm{d}x)\mathrm{d}\l.
\eea
We will refer to $\frac{\d F}{\d m}$ as the linear functional derivative.
There is at most one $\frac{\d F}{\d m}$, up to a constant shift, satisfying \eqref{C1derv}. To avoid the ambiguity, we impose
\beaa
\int_{\dbR^d} \frac{\d F}{\d m}(m, x)m(\mathrm{d}x) =0.
\eeaa
If $(m,x)\mapsto \frac{\d F}{\d m}(m,x)$ is continuously differentiable in $x$, we define its intrinsic derivative $D_m F: \cP(\dbR^d)\times \dbR^d\rightarrow \dbR^d$ by
\beaa
D_m F(m,x) = \nabla \left( \frac{\d F}{\d m}(m,x) \right).
\eeaa
In this paper $\nabla$ always denotes the gradient in the variable $x\in \dbR^d$.
\begin{eg}\label{eg:linear}
	If $F(m):= \int_{\dbR^d} \phi(x) m(\mathrm{d}x)$ for some bounded continuous function $\phi:\dbR^d \rightarrow\dbR$, we have $\frac{\d F}{\d m}(m,x) = \phi(x)$ and $D_m F(m,x) = \dot \phi(x)$.
\end{eg}

It is useful to see what intrinsic measure derivative look like in the special case when we consider empirical measures 
\[
m^N := \frac1N \sum_{i=1}^N \delta_{x^i},\,\,\,\text{where } x^i \in \mathbb R^d.
\]
Then one can define $ F^N:(\mathbb R^d)^N \rightarrow \mathbb R $ as 
$F^N(x^1,\ldots,x^N)  = F(m^N )$.  
From~\cite[Proposition 3.1]{chassagneux2014probabilistic} we know that that if $F \in \mathcal C^1$ then $F^N \in C^1$ and for any $i=1,\ldots,N$ and $(x^1,\ldots,x^N)\in (\mathbb R^d)^N$ it holds that
\begin{equation}
	\label{eq projection derivative}
	\partial_{x^i} F^N ( x^1,\ldots,x^N) =  \frac{1}{N} D_m F \left(m^N, x^i\right) \,.
\end{equation}

We remark that for notational simplicity in the proofs the constant $C>0$ can be different from line to line. 

\section{Main Results}
\label{sec main}

The objective of this paper is to study the minimizer(s) of a convex function $F:\cP(\dbR^d)\rightarrow\dbR$.
\begin{assum}\label{assum:energy}
	Assume that $F\in \cC^1$ is convex and bounded from below.
\end{assum}

\no 
Instead of directly considering the minimization $\min_m F(m)$, we propose to first study the regularized version, namely, the minimization of the free energy function:
\begin{equation}\label{eq:optimization}
	\min_{m\in \cP(\dbR^d)}  V^{\sigma}(m),\q\mbox{where}\q V^{\sigma}(m):= F(m) + \frac{\sigma^2}{2}H(m), \q\mbox{for all}\q m\in\mathcal{P}(\mathbb{R}^d),
\end{equation}
where $ H:\cP(\dbR^d)\rightarrow [0,\infty] $ is the relative entropy (Kullback--Leibler divergence) with respect to a given Gibbs measure in $ \mathbb{R}^d $, namely,
\begin{equation*}
	H(m) := \int_{\mathbb{R}^d}m(x)\log\left( \frac{m(x)}{g(x)} \right)\mathrm{d}x,
\end{equation*}
where \[
g(x) = e^{-U(x)}\,\,\,\text{with $U$ s.t.}\,\,\,\int_{\mathbb R^d} e^{-U(x)} \, \mathrm{d}x = 1\,,
\]
is the density of the Gibbs measure and the function $U$ satisfies the following conditions.

\begin{assum}\label{assum:U}
	The function $U:\dbR^d\rightarrow\dbR$ belongs to $C^\infty$. 
	Further,
	\begin{enumerate}[i)]
		
		\item there exist constants $C_U>0$ and $C_U'\in \dbR$ such that
		\bea\label{assum:diss}
		\nabla U(x) \cd  x \ge C_U|x|^2 +C_U' \q\mbox{for all}~~x\in \dbR^d\,.
		\eea
		
		\item $\nabla U$ is Lipschitz continuous.
		
	\end{enumerate}
\end{assum}

\no Immediately, we obtain that there exist $0\le C'\le C$ such that for all $x\in \dbR^d$
\beaa
C'|x|^2 -C \le U(x) \le C(1+|x|^2), 
\q  |\D U(x) |\le C.
\eeaa
A typical choice of $g$ would be the density of the $d$-dimensional standard Gaussian distribution.
We recall that such relative entropy $H$ has the properties:
it is strictly convex when restricted to measures absolutely continuous with $g$, 
it is  weakly lower semi-continuous and its sub-level sets  are compact.
For more details, we refer the readers to the book \cite[Section 1.4]{DE97}. The original minimization and the regularized one is connected through the following $\Gamma$-convergence result.

\begin{prop}
	\label{prop:Gamma}
	Assume that $F$ is continuous in the topology of weak convergence. 
	Then the sequence of functions $V^\sigma = F + \frac{\sigma^2}{2}H$  $\Gamma $-converges to $ F $ when $ \sigma\downarrow0 $. In particular, given the minimizer $m^{*,\si}$ of $V^\si$, we have
	\beaa
	\limsup_{\si\rightarrow 0}  F(m^{*,\si}) ~=~ \inf_{m\in \cP_2(\dbR^d)} F(m). 
	\eeaa
\end{prop}

\no It is a classic property of $\Gamma$-convergence that 
every cluster point of $\big(\argmin\limits_m V^\si(m)\big)_\si$ is a  minimizer of $F$.

The following theorem shows that we can control the error between the finite and infinite-dimensional optimization problems.
It generalises  
\cite[Proposition 2.1]{mei2018mean} to an arbitrary (smooth) functions of measure. It is an extension of the result from \cite[Theorem 2.11]{chassagneux2019weak}.

\begin{thm} \label{th static particles}
	We assume that the $2$nd order linear functional derivative of $F$ exists, is jointly continuous in both variables and that there is $L > 0$ such that
	for any random variables $\eta_1$, $\eta_2$ such that $\mathbb E[|\eta_i|^2] < \infty$, $i=1,2$, it holds that 
	\begin{align} \label{as int}
		\mathbb E\left[ \sup_{\nu \in \mathcal P_2(\mathbb R^d)} \left|\frac{\delta F}{\delta m}(\nu,\eta_1)\right|\right] 
		+ 
		\mathbb E\left[ \sup_{\nu \in \mathcal P_2(\mathbb R^d)} \left|\frac{\delta^2 F}{\delta m^2}(\nu,\eta_1,\eta_2)\right|\right] \leq L\;
	\end{align}   
	If there is an $m^\ast \in \mathcal P_2(\mathbb R^d)$ such that $F(m^\ast) = \inf_{m\in \mathcal P_2(\mathbb R^d)} F(m)$ then
	we have that
	\[
	\left| \inf_{(x_i)_{i=1}^N\subset \mathbb R^d} F\left(\frac1N \sum_{i=1}^N \delta_{x_i}\right) -  F(m^\ast) \right| \leq \frac{2L}N \,.
	\]
\end{thm}

Moreover, when the relative entropy $H$ is strictly convex, then so is the function $V$, and thus the minimizer $\argmin_{m\in\mathcal P(\mathbb R^d)} V(m)$, if exists, must be unique. It can be characterized by the following first order condition. 

\begin{prop}\label{prop:firstorder}
	Under Assumption \ref{assum:energy} and \ref{assum:U}, the function $ V^\sigma $ has a unique minimizer absolutely continuous with respect to Lebesgue measure $ \ell  $, and belonging to $\cP_2 (\dbR^d)$. Moreover, $m^*\in\cP_2(\dbR^d) =\argmin\limits_{m\in \cP(\dbR^d)} V^\sigma(m) $ if and only if $m^*$ is equivalent to Lebesgue measure and 
	\begin{equation}\label{InvariantSet}
		\frac{\delta F}{\delta m}(m^*,\cd) + \frac{\sigma^2}{2}\log(m^*) + \frac{\sigma^2}{2}U~\text{ is a constant, }\ell-a.s.,
	\end{equation}
	where we abuse the notation, still denoting by $m^*$ the density with respect to Lebesgue measure. 
\end{prop}

Further, we are going to  approximate the minimizer of $V^\si$, using the marginal laws of the solution to the  upcoming mean field Langevin equation.
Let $ \sigma\in\mathbb{R}_+ $ and consider the following McKean--Vlasov SDE:
\begin{equation}\label{Langevin}
	\mathrm{d}X_t = -\left(D_mF(m_t, X_t) + \frac{\sigma^2}{2}\nabla U(X_t)\right)\mathrm{d}t + \sigma \mathrm{d}W_t,
\end{equation}
where $ m_t $ is the law of $ X_t $ and $ (W_t)_{t\geq0} $ is a standard $ d $-dimensional Brownian motion.
\begin{rem}
	\begin{enumerate}[i)]
		\item Let $F(m) = \int_{\dbR^d} f(x) m(\mathrm{d}x)$ for some  function $f$ in $C^1(\mathbb R^d,\mathbb R)$. We know that $D_m F(m,x) = \nabla f(x)$. Hence with this choice of $F$ and entropy regulariser with respect to the Lebesgue measure, the dynamics \eqref{Langevin} becomes the standard overdamped Langevin equation \eqref{classicLangevin}. 
		
		\item 
		If the Gibbs measure is chosen to be a standard Gaussian distribution, the potential of the drift of \eqref{Langevin} becomes $ F(m)+ \frac{\si^2}{4}\int_{\dbR^d}|x|^2 m(\mathrm{d}x) $.  This shares the same spirit as ridge regression.
	\end{enumerate}
	
\end{rem}

\begin{assum}\label{A1}
	Assume that the intrinsic derivative $ D_mF:\mathcal{P}(\mathbb{R}^d) \times\mathbb{R}^d\to\mathbb{R}^d $ of the function $ F:\mathcal{P}(\mathbb{R}^d)\to\mathbb{R} $ exists and  satisfies the following conditions:
	\begin{enumerate}[i)]
		\item $D_m F$ is bounded and  Lipschitz continuous, i.e. there exists $ C_F>0 $ such that for all $x,x\in \dbR^d$ and $m,m'\in \cP_2(\dbR^d)$
		\begin{equation}\label{loclip}
			|D_mF(m,x) - D_mF(m',x')|\leq
			C_F \big( |x - x'| + \cW_2(m, m') \big)
		\end{equation}
		
		\item $ D_mF(m,\cdot)\in\mathcal{C}^\infty(\mathbb{R}^d) $ for all $m\in \cP(\dbR^d)$;
		
		\item $\nabla D_m F: \cP(\dbR^d)\times\dbR^d \rightarrow \dbR^d\times \dbR^d$ is jointly continuous. 
	\end{enumerate}
\end{assum}

The well-posedness of the McKean--Vlasov SDE~\eqref{Langevin} under Assumption~\ref{assum:U} and~\ref{A1} on the time interval $[0,t]$, for any $t$, is well known, see e.g. Snitzman~\cite{sznitman1991topics}, so the proof of the following proposition is omitted.

\begin{prop}\label{prop:wellpose}
	Under Assumption \ref{assum:U} and \ref{A1} the mean field Langevin SDE \eqref{Langevin} has a unique strong solution, if $m_0\in \cP_2(\dbR^d)$.  
	Moreover, the solution is stable with respect to the initial law, that is, given $m_0, m'_0\in \cP_2(\dbR^d)$, denoting by $(m_t)_{t\in\dbR^+}, (m'_t)_{t\in\dbR^+}$ the marginal laws of the corresponding solutions to \eqref{Langevin}, we have for all $t>0$ there is a constant $C>0$ such that
	\beaa
	\cW_2(m_t, m'_t) ~\le~ C \cW_2(m_0, m'_0).
	\eeaa
\end{prop}
We shall prove the process $\big(V^\si(m_t)\big)_t$ is decreasing and satisfies the following dynamic.

\begin{thm}\label{thm:Vdecrease}
	Let $m_0\in \cP_2(\dbR^d). $Under Assumption \ref{assum:U} and \ref{A1},  we have for any $t>s>0$
	\bea\label{eq: measureFlow}
	V^\sigma(m_t) - V^\sigma(m_s) = -\int_{s}^{t}\int_{\mathbb{R}^{d}}\left| D_mF(m_r,x) + \frac{\sigma^2}{2}\frac{\nabla m_r}{m_r}(x) + \frac{\sigma^2}{2}\nabla U(x)\right|^2m_r(x)\mathrm{d}x\mathrm{d}r.
	\eea
\end{thm}

\begin{rem}
	
	In order to prove \eqref{eq: measureFlow}, we use the generalized It\^o calculus as the main tool. Alternative proofs of Theorem \ref{thm:Vdecrease} can be obtained under the comparable assumptions using the theory of gradient flows, see e.g. the monograph \cite{ambrosio2008gradient}. Also note that our path-wise argument shares the spirit with the recent work \cite{KST19}, which recovers the results of gradient flows for probability measure on the Euclidean space using the It\^o calculus but only for linear functional $F$.
	%
	%
\end{rem}

\no Formally, there is a clear connection between the derivative $\frac{\mathrm{d} V^\si(m_t)}{\mathrm{d}t}$ in \eqref{eq: measureFlow} and the first order condition \eqref{InvariantSet}, and it is revealed by the following main theorem.

We call a measure $ m$ an invariant measure of \eqref{Langevin}, if ${\rm Law}(X_t)= m$ for all $t\ge 0$ .

\begin{thm}\label{thm:convergence}
	Let Assumption \ref{assum:energy}, \ref{assum:U} and  \ref{A1} 
	hold true and  $m_0\in \cup_{p>2}\cP_p(\dbR^d)$. Denote by $(m_t)_{t\in \dbR^+}$   the flow of marginal laws of the solution to \eqref{Langevin}. There exists an invariant measure of  \eqref{Langevin} equal to  $m^*:=\argmin\limits_m V^\si(m)$, and $\lim_{t\rightarrow\infty} \cW_2(m_t, m^*)=0$.
\end{thm}

\begin{rem}
	As mentioned, the main  contribution of this paper is to prove the $\cW_2$-convergence of the marginal laws of \eqref{Langevin} towards the invariant measure under the mild conditions (Assumption \ref{assum:energy}, \ref{assum:U} and  \ref{A1}). Note that it is possible to obtain exponential convergence result with extra conditions on the coefficients. More precisely, given the constants $C_F, \rho_F, C_U$ such that
	\beaa
	& |D_mF(m,x) - D_mF(m',x')|\leq
	C_F |x - x'| + \rho_F \cW_1(m, m')\q\mbox{for all $x,x'\in \dbR^d$ and $m,m'\in \cP(\dbR^d)$},&\\
	& (x - x' ) \cd  (\nabla U(x) - \nabla U(x')) \ge C_U |x-x'|^2 \q\mbox{for $|x-x'|$ big enough },&
	\eeaa
	Eberle, Guillin and Zimmer \cite{eberle2019quantitative} proved that there exists a  constant $\gamma$  depending on $\sigma$, $C_F$ and $C_U$ such that
	\bea\label{eq:exp_conv}
	\cW_1(m_t, m'_t) \le C e^{-(\g -\rho_F) t } \cW_1(m_0, m'_0),
	\eea
	where $(m'_t)_{t\in \dbR^+}, (m'_t)_{t\in \dbR^+}$ are the flows of marginal laws of the solutions to \eqref{Langevin} with the initial law $m_0,m'_0$, respectively. In particular, 
	\begin{itemize}
		\item the result \eqref{eq:exp_conv} only implies the exponential contraction provided that $\rho_F$ is small enough, that is, the mean field dependence must be small;
		
		\item the constant $\g$ is increasing in $\si$ and $C_U$, so $\g$ is big only if $\si$ or/and $C_U$ are large, that is, the optimization \eqref{eq:optimization} is over-regularized.
	\end{itemize}
\end{rem}

\section{Application to Gradient Descent of Neural Networks}
\label{sec applications and sgd}

Before proving the main results, we shall first apply them to study the minimization over a neural network. In particular, in Corollary \ref{cor:NNworks} we shall show that the marginal laws of the corresponding mean-field Langevin dynamics converge to the optimal weight of the neural network with $1$-hidden layer. 

Fix a locally Lipschitz function $\varphi:\mathbb R \to \mathbb R$ and
for $l\in \mathbb N$ define $\varphi^l : \mathbb R^l \to \mathbb R^l$ as the function given, for $z=(z_1,\ldots,z_l)^\top$ by 
$\varphi^l(z) = (\varphi(z_1),\ldots, \varphi(z_l))^\top$.
We fix $L\in \mathbb N$ (the number of layers), $l_k \in \mathbb N$, $k=0,1,\ldots L-1$ (the size of input to layer $k$) and $l_L \in \mathbb N$ (the size of the network output). 
A fully connected artificial neural network is then given by $\Psi = ((\alpha^1,\beta^1), \ldots, (\alpha^L, \beta^L)) \in \Pi$,
where, for $k=1,\ldots,L$, we have real $l^k\times l^{k-1}$ matrices $\alpha^k$ and real $l^k$-dimensional
vectors $\beta^k$. 	
We see that $\Pi = (\mathbb R^{l^1\times l^0}\times \mathbb R^{l^1}) \times (\mathbb R^{l^2\times l^1}\times \mathbb R^{l^2})\times\cdots\times (\mathbb R^{l^L\times l^{L-1}}\times \mathbb R^{l^L})$.
The artificial neural network defines a reconstruction function $\mathcal R\Psi : \mathbb R^{l^0} \to \mathbb R^{l^L}$ given recursively, for $z_0 \in \mathbb R^{l^0}$, by 
\[
(\mathcal R\Psi)(z^0) = \alpha^L z^{L-1} + \beta^L\,, \,\,\,\, 
z^k = \varphi^{l^k}(\mathcal \alpha^k z^{k-1} +\beta^k)\,,k=1,\ldots, L-1\,.
\]
If for each $k=1,\ldots, L-1$ we write $\alpha_i^k$, $\beta_i^k$ to denote the $i$-th row of the matrix $\alpha^k$ and vector $\beta^k$ respectively 
then we can write the reconstruction of the network equivalently as
\begin{equation}
	\label{eq ann realization 2}
	(\mathcal R\Psi)(z^0)_i = \alpha_{i}^L \cdot z^{L-1} + \beta_i^L\,, \,\,\,\, 
	(z^k)_i = \varphi(\alpha_i^k \cdot z^{k-1} + \beta_i^k)\,,k=1,\ldots, L-1\,.	
\end{equation}
We note that the number of parameters in the network is 
$
\sum_{i=1}^L (l_{k-1}l_k + l_k )\,.
$


Given a potential function $\Phi$ and training data $(y^j,z^j)_{j=1}^N$, $(y_j,z_j)\in \mathbb R^d$ one approximates
the optimal parameters by finding
\begin{equation}
	\label{eq:mincomplex}
	\argmin_{\Psi \in \Pi} \frac{1}{N}\sum_{j=1}^N \Phi \Big(y^j - (\mathcal R \Psi)(z^j) \Big)\,.
\end{equation}
This is a non-convex minimization problem, so in general hard to solve. Theoretically, the following universal representation theorem ensures that the minimum value should attain $0$, provided that $y=f(z)$ with a continuous function $f$.
\begin{thm}[Universal Representation Theorem]
	If an activation function $\f$ is bounded, continuous and non-constant, then 
	for any compact set $K\subset \dbR^d$ the set 
	\[
	\left\{(\mathcal R \Psi): \dbR^d \rightarrow\dbR : (\mathcal R \Psi)~\mbox{given by~\eqref{eq ann realization 2} with $L=2$ for some}~ n\in \mathbb N, \alpha^2_j, \beta^1_j\in \dbR, \alpha^1_j\in \dbR^d, j=1,\ldots, n \right\}
	\]
	is dense in $C(K)$. 
\end{thm}

\no For an elementary proof, we refer the readers to  \cite[Theorem 2]{Hor91}.

\subsection{Fully connected 1-hidden layer neural network}
\label{sec fully connected 1-hidden layer}


Take $L=2$, fix $d\in \mathbb N$ and $n\in \mathbb N$
and consider the following 1-hidden layer neural network for approximating
functions from $\mathbb R^d$ to $\mathbb R$:
let $l_0 = d$, let $l_1 = n$, let $\beta^2 = 0 \in \mathbb R$, $\beta^1 = 0\in \mathbb R^n$,  $\alpha^1 \in \mathbb R^{n \times d}$. 
We will denote, for $i\in\{1,\ldots,l^0\}$, its $i$-th row by  $\alpha^1_i \in \mathbb R^{1 \times d}$.
Let $\alpha^2 = (\frac{c_1}{n},\cdots,\frac{c_{n}}{n})^\top$, where $c_i \in \mathbb R$. 
The neural network is $\Psi^n = \Big((\alpha^1,\beta^1), (\alpha^2, \beta^2) \Big)$ (where we emphasise the that the size of the hidden layer is $n$).
For $z\in \mathbb R^{l^0}$, its reconstruction can be written as
\[
(\mathcal R \Psi^n)(z) = \alpha^2 \varphi^{l^1}( \alpha^1  z   )
= \frac{1}{n}\sum_{i=1}^{n} c_i \varphi( \alpha^1_i \cd z   )\,.
\]	

\no The key observation is to note that,  due to law of large numbers (and under appropriate technical assumptions)  
$\frac{1}{n}\sum_{j=1}^n c_j\f (\alpha^1_j\cd z) \rightarrow \mathbb{E}^{m}[B\f(A\cdot z)]$ as $n \to \infty$, where $m$ is the law of the pair of random variables $(B, A)$ and $\dbE^m$ is the expectation under the measure $m$. Therefore, another way (indeed a more intrinsic way regarding to the universal representation theorem) to formulate the minimizaiton problem \eqref{eq:mincomplex} is:
\beaa
\min_{m \in \mathcal P(\dbR^d\times \dbR)}  \tilde F(m)\,, \q 
\mbox{where}\q  \tilde F(m) := \int_{\mathbb{R}^{d}}\Phi\big(y-\mathbb{E}^{m}[B\f(A\cdot z)]\big)\,\nu(\mathrm{d}z,\mathrm{d}y)\,.
\eeaa
For technical reason, we introduce a truncation function $\ell:\dbR\rightarrow K$ ($K$ denotes again some compact set), and consider the truncated version of the minimization:
\beaa
F(m):= \int_{\mathbb{R}^{d}}\Phi\big(y-\mathbb{E}^{m}[\ell(B)\f(A\cdot z)]\big)\,\nu(\mathrm{d}z,\mathrm{d}y).
\eeaa
It is crucial to note that in the reformulation the objective function $F$ becomes a convex function on $\cP(\dbR^d)$, provided that $\Phi$ is convex. 

\begin{assum}\label{assum:nn}
	We apply the following assumptions on the coefficients $\Phi, \mu, \f, \ell$:
	\begin{enumerate}[i)]
		\item the function $\Phi$ is convex, smooth and $0=\Phi(0) =\min_{a\in \dbR} \Phi(a)$;
		\item the data measure $\mu$ is of compact support;
		\item the truncation function $\ell\in C^\infty_b (\dbR^d)$ such that $\dot{\ell}$ and $\ddot\ell$ are bounded;
		\item the activation function $\f \in C^\infty_b (\dbR^d)$ such that $\dot{\f}$ and $\ddot\f$ are bounded.
	\end{enumerate}
\end{assum}

\begin{cor}\label{cor:NNworks}
	Under Assumption \ref{assum:nn}, the function $F$ satisfies Assumption  \ref{assum:energy}, \ref{A1}.  In particular, with a Gibbs measure of which the function $U$ satisfies Assumption \ref{assum:U}, the corresponding mean field Langevin equation \eqref{Langevin} admits a unique strong solution, given $m_0\in \cP_2(\dbR^d)$.  Moreover,  the flow of marginal laws of the solution, $(m_t)_{t\in \dbR^+}$, satisfies
	\beaa
	\lim_{t\rightarrow +\infty}\cW_2 \left(m_t , \argmin_{m\in \cP(\dbR^d)} V^\si(m)\right) =0.
	\eeaa
\end{cor}
\begin{proof}
	Let us define, for $x=(\beta,\alpha)\in \mathbb R^d$, $\beta \in \mathbb R$, $\alpha \in \mathbb R^{d-1}$ and $z\in \mathbb R^{d-1}$ the function $\hat \varphi(x,z) := \ell(\beta)\varphi(\alpha \cdot z)$.
	Then
	\beaa
	&\frac{\d F}{\d m}(m,x) 
	= - \int_{\dbR^d} \dot \Phi\big(y- \dbE^m [\hat \varphi(X, z)] \big) \hat \varphi(x,z)\, \nu(\mathrm{d}z,\mathrm{d}y) \\
	& \mbox{and}\q D_m F (m,x) 
	= - \int_{\dbR^d } \dot \Phi\big(y -\dbE^m [\hat \varphi(X, z)]\big) 
	\nabla \hat \varphi(x,z)\,\nu(\mathrm{d}z,\mathrm{d}y)\,.
	\eeaa
	Then it becomes straightforward to verify that $F$ satisfies both Assumption \ref{assum:energy}, \ref{A1}.
	The rest of the result is direct from Proposition \ref{prop:wellpose} and Theorem \ref{thm:convergence}.
\end{proof}

\subsection{Gradient Descent}
\label{sec gradient descent}

Consider independent random variables $(X_0^i)_{i=1}^N$, $X_0^i \sim m_0$ and independent Brownian motions $(W^i)_{i=1}^N$. By approximating the law of the process \eqref{Langevin} by its empirical law we arrive at the following interacting particle system 
\begin{equation}\label{eq particles}
	\begin{cases}
		\mathrm{d}X^i_t =& -  \Big(D_m F(m^N_t , X^i_t ) + \frac{\sigma^2}{2} \, \nabla U(X^i_t) \Big) \mathrm{d}t +\si \mathrm{d}W^i_t\,\,\, i=1,\ldots, N\,, \\
		m^N_t =& \frac1N \sum_{i=1}^N \delta_{X_t^i}\,.
	\end{cases}
\end{equation}
Note that particles $(X^i)_{i=1}^N$ are not independent, but their laws are exchangeable. Recall the link between partial derivatives and measure derivative 
given by~\eqref{eq projection derivative} and for any $(x^1,\ldots,x^N)\in (\mathbb R^d)^N$ let $F^N(x^1,\ldots, x^N) = F\left(\frac1N\sum_{i=1}^N \delta_{x^i}\right)$.
Then
\[
\mathrm{d}X^i_t = - \Big(N \partial_{x_i} F^N(X^1_t,\ldots,X^N_t) + \frac{\sigma^2}{2} \, \nabla U(X^i_t) \Big) \mathrm{d}t +\si \mathrm{d}W^i_t\,.
\]
Let us define, for $x=(\beta,\alpha)\in \mathbb R^d$, $\beta \in \mathbb R$, $\alpha \in \mathbb R^{d-1}$ and $z\in \mathbb R^{d-1}$ the function $\hat \varphi(x,z) := \ell(\beta)\varphi(\alpha \cdot z)$.
Then for $(x^i)_{i=1}^N$ we have
\[
F^N(x) = \int_{\mathbb{R}^{d}}\Phi\bigg(y- \frac1N \sum_{j=1}^N \hat \varphi(x^j,z)\bigg)\,\nu(\mathrm{d}z,\mathrm{d}y).
\]
Hence
\[
\partial_{x^i} F^N (x^1,\ldots,x^N) 
= - \frac1N \int_{\dbR^d} \dot \Phi\bigg(y- \frac1N \sum_{j=1}^N\hat \varphi(x^j, z) \bigg) \nabla \hat \varphi(x^i,z) \nu(\mathrm{d}z,\mathrm{d}y)\,,
\]
where we denote for all $z\in \dbR^{d-1}$
\[
\nabla \hat \varphi(x^i,z) = \nabla_{(\beta^i,\alpha^i)}[\ell(\beta^i)\varphi(\alpha^i z)] 
= \Bigl( \begin{smallmatrix} \dot\ell (\b^i) \f (\a^i\cd z)\\ \ell(\b^i) \dot\f (\a^i\cd z) z\end{smallmatrix}\Bigr)\,. 
\]
We thus see that~\eqref{eq particles} corresponds to 
\[
\mathrm{d}X^i_t = \left( \int_{\dbR^d} \dot \Phi\bigg(y- \frac1N \sum_{j=1}^N\hat \varphi(X_t^j, z) \bigg) \nabla \hat \varphi(X_t^i,z)  \nu(\mathrm{d}z,\mathrm{d}y) - \frac{\sigma^2}{2}  \nabla U(X^i_t) \right) \mathrm{d}t +\si \mathrm{d}W^i_t.
\]
This is classical Langevin dynamics
\eqref{classicLangevin} on $(\mathbb R^d)^N$. 
One may reasonably expect that the a version of Theorem~\ref{th static particles} can be proved in this dynamical setup. 
This has been done for finite time horizon problem in \cite{chassagneux2019weak}. The extension to the infinite horizon requires uniform in time regularity of the corresponding PDE on Wasserstein space $(\mathcal W_2, \mathcal P_2)$ and we leave it for a future research. However rate for uniform propagation of chaos in $\mathcal W_1$ under structural condition on the drift has been proved in \cite{durmus2018elementary}. We also remark that for the implementable algorithm one works with time discretisation of \eqref{eq particles} and, at least for the finite time, the error bounds are rather well understood \cite{bossy1997stochastic,bossy2002rate,malrieu2003convergence,szpruch2017iterative,szpruch2019antithetic}.  

For a fixed time step $\tau > 0$ fixing a grid of time points $t_k = k\tau$, $k=0,1,\ldots$ we can then write the explicit Euler scheme
\begin{multline*}
	X^{\tau,i}_{t_{k+1}} - X^{\tau,i}_{t_k} \\
	= \left( \int_{\dbR^d} \dot \Phi\bigg(y- \frac1N \sum_{j=1}^N\hat \varphi(X_{t_k}^{\tau,j}, z) \bigg) \nabla \hat \varphi(X_{t_k}^{\tau,i},z)  \nu(\mathrm{d}z,\mathrm{d}y) - \frac{\sigma^2}{2} \nabla U(X^{\tau,i}_{t_k}) \right) \tau +\si (W^i_{t_{k+1}} - W^i_{t_k})\,.
\end{multline*}
To relate this to the gradient descent algorithm we consider the case where 
we are given data points $(y_m, z_m)_{n\in \mathbb N}$ which are i.i.d. samples from $\nu$. 
If the loss function $\Phi$ is simply the square loss then a version of the (regularized) gradient descent algorithm for the evolution of parameter $x_k^i$ will simply read as
\[
x^i_{k+1} = x^i_k + 2\tau\left( \bigg(y_k - \frac1N \sum_{j=1}^N\hat \varphi(x_k^j, z_k) \bigg) \nabla \hat \varphi(x_k^i, z^k)  - \frac{\sigma^2}{2} \, \nabla U(x^i_k) \right) + \si \sqrt{\tau}\xi_k^i\,,
\]
with $\xi_k^i$ independent samples from $N(0,I_d)$.

\section{Free Energy Function}
\label{sec free energy fn}

In this section, we study the properties concerning the minimizer of the free energy function  $V^\si$.
First, we prove that $V^\si$ is an approximation of $F$ in the sense of $\Gamma$-convergence. 

\begin{proof}[Proof of Proposition \ref{prop:Gamma}]
	Let $ (\sigma_n)_{n\in\mathbb{N}} $ be a positive sequence decreasing to $ 0 $. 
	On the one hand, since $ F $ is continuous and $ H(m)\geq0 $, for all $ m_n \rightarrow m $, we have
	\beaa
	\liminf_{n\to+\infty}V^{\sigma_n}(m_n)\geq \lim_{n\to+\infty}F(m_n) = F(m).
	\eeaa
	On the other hand, given $ m\in\mathcal{P}_2(\mathbb{R}^d) $, since the function 
	\bea\label{eq:funh}
	h(x):=x\log(x) 
	\eea
	is convex, it follows  Jensen's inequality that
	\begin{equation*}
		\int_{\mathbb{R}^{d}}h(m\ast f_n)\mathrm{d}x \leq \int_{\dbR^d} \int_{\dbR^d} h\big(f_n(x-y)\big) m(\mathrm{d}y) \mathrm{d}x = \int_{\dbR^d} h\big(f_n(x)\big) \mathrm{d}x 
		= \int_{\dbR^d} h\big(f(x)\big) \mathrm{d}x - d \log(\si_n),
	\end{equation*}
	where $f$ is the heat kernel and $ f_n(x) = \si_n^{-d}f(x/\si_n) $. Besides, we have
	\beaa
	\int_{\mathbb{R}^{d}} (m\ast f_n )\log(g)\mathrm{d}x 
	= - \int_{\mathbb{R}^{d}}m(\mathrm{d}y) \int_{\mathbb{R}^{d}}  f_n(x)U(x-y)\mathrm{d}x \ge - C\Big(1+\int_{\mathbb{R}^{d}} |y|^2 m(\mathrm{d}y) \Big).
	\eeaa
	The last inequality is due to the quadratic growth of $U$. Therefore
	\begin{equation}\label{eq:limsupGamma}
		\limsup_{n\to+\infty}V^{\sigma_n}(m\ast f_n) 
		\leq F(m) + \limsup_{n\to+\infty} \frac{\si_n^2}{2}\Big\{ \int_{\mathbb{R}^{d}}h(m\ast f_n)\mathrm{d}x - \int_{\mathbb{R}^{d}} (m\ast f_n )\log(g)\mathrm{d}x  \Big\} \le F(m).
	\end{equation}
	In particular, given a minimizer $m^{*,\si}$ of $V^\si$,  by \eqref{eq:limsupGamma} we have
	\beaa
	\limsup_{n\rightarrow \infty}  F(m^{*,\si_n}) 
	~\le ~  \limsup_{n\rightarrow \infty}  V^\si(m^{*,\si_n}) 
	~\le ~ \limsup_{n\to+\infty}V^{\sigma_n}(m\ast f_n) 
	~\le~  F(m), \q\mbox{for all}\q m\in \cP_2(\dbR^d).
	\eeaa
	
\end{proof}

\begin{proof}[Proof of Theorem~\ref{th static particles}] 
	Let $\mu \in \mathcal P_2(\mathbb R^d)$ be arbitrary.
	Let $(X_i)_{i=1}^N$ be i.i.d. with law $\mu$.
	Let $\mu_N = \frac1N \sum_{i=1}^N {\delta_{X_i}}$
	and $m^N_t = \mu + t (\mu_N-\mu)$, $t \in [0,1]$.
	Further let  $(\tilde{X}_i)_{i=1}^N$ be consider i.i.d., independent of $(X_i)_{i=1}^N$ with law $\mu$.
	
	By the definition of linear functional derivatives, we have
	\begin{align}
		\mathbb E [ F(\mu_N)] -  F(\mu) &= 
		\mathbb E \left[\int_0^1 \int_{\mathbb R^d} \frac{\delta F}{\delta m}(m^N_t, v)\, (\mu_N-\mu)(\mathrm{d}v) \mathrm{d} t\right] \nonumber 
		\\
		&=\int_0^1  \frac1N\sum_{i=1}^N\left( \mathbb E \left[ \frac{\delta F}{\delta m}(m^N_t, X_1)\right] -\mathbb E\left[ \frac{\delta F}{\delta m}(m^N_t, \tilde{X}_1)\right] \right) \mathrm{d} t \nonumber 
		\\
		&=\int_0^1 \mathbb E \left[ \frac{\delta F}{\delta m}(m^N_t,X_1) -\frac{\delta F}{\delta m}(m^N_t, \tilde{X}_1)\right] \mathrm{d} t\,.
	\end{align}
	We introduce the (random) measures
	\begin{align*}
		\tilde{m}^{N}_{t} :=  {m}^{N}_{t} + \frac{t}{N} (\delta_{\tilde{X}_1} - \delta_{{X}_1}) \quad \text{and} \quad {m}^{N}_{t,t_1}:= (\tilde{m}^{N}_{t}-m^{N}_t)t_1 + m^{N}_t, 
		\quad t,t_1 \in [0,1],
	\end{align*}
	and notice that due to independence of $(X_i)_{i=1}^N$ and $(\tilde{X}_i)_{i=1}^N$ we have that
	\begin{align*}
		\mathbb E \left[\frac{\delta F}{\delta m}(\tilde{m}^{N}_{t}, \tilde{X}_1)\right] = \mathbb E\left[\frac{\delta F}{\delta m}({m}^{N}_{t},X_1)\right].
	\end{align*}
	Therefore, 
	\begin{equation} \label{eq fmu}
		\begin{split}
			\mathbb E [ F(\mu_N) - F(\mu)]  = &  \int_0^1  \mathbb E \left[ \frac{\delta F}{\delta m}(\tilde{m}^{N}_{t},\tilde{X}_1)-\frac{\delta F}{\delta m}(m^N_t, \tilde{X}_1)\right]  \mathrm{d} t  \\
			= &  \int_0^1 \mathbb E \left[ \int_0^1  \int_{\mathbb R^d} \frac{\delta^2 F}{\delta m^2}( {m}^{N}_{t,t_1},\tilde{X}_1,y_1) (\tilde{m}^{N}_{t} - {m}^{N}_{t}) (\mathrm{d}y_1) \, \mathrm{d}t_1 \right]  \mathrm{d} t  \\
			= & \frac1{N}  \mathbb E \left[ \int_0^1 \int_0^1  \int_{\mathbb R^d} t \frac{\delta^2 F}{\delta m^2}( {m}^{N}_{t,t_1},\tilde{X}_1,y_1) (\delta_{\tilde{X}_1} - \delta_{{X}_1}) (\mathrm{d}y_1) \, \mathrm{d}t_1 \mathrm{d} t \right]. 
		\end{split}
	\end{equation}
	To conclude, we observe that
	\begin{equation*}
		\begin{split}
			& \mathbb E\left[\bigg| \int_{\mathbb R^d}\frac{\delta^2 F}{\delta m^2}( {m}^{N}_{t,t_1})(\tilde{X}_1,y_1) (\delta_{\tilde{X}_1} - \delta_{{X}_1}) (\mathrm{d}y_1)\bigg|\right]\\
			& = \mathbb E\left[ \bigg| \int_{\mathbb R^d}\frac{\delta^2 F}{\delta m^2}( {m}^{N}_{t,t_1})(\tilde{X}_1,y_1) \delta_{\tilde{X}_1} (\mathrm{d}y_1) -  \int_{\mathbb R^d}\frac{\delta^2 F}{\delta m^2}( {m}^{N}_{t,t_1})(\tilde{X}_1,y_1) \delta_{{X}_1}(\mathrm{d}y_1)\bigg|\right]\\
			& \leq \mathbb E\left[\sup_{\nu \in \mathcal P_2(\mathbb R^d)} \left|\frac{\delta^2 F}{\delta m^2}( \nu)(\tilde{X}_1,\tilde{X}_1)\right|
			+ \sup_{\nu \in \mathcal P_2(\mathbb R^d)} \left|\frac{\delta^2 F}{\delta m^2}( \nu)(\tilde{X}_1,{X}_1)\right|\right]
			\leq 2L\,,	
		\end{split}	
	\end{equation*}
	by~\eqref{as int}. 
	We have thus shown that for all $\mu \in \mathcal P_2(\mathbb R^d)$,
	for all i.i.d. $(X_i)_{i=1}^N$ with law $\mu$ and with $\mu_N = \frac1N \sum_{i=1}^N {\delta_{X_i}}$
	it holds that
	\begin{equation}
		\label{eq fmu2}
		|\mathbb E [ F(\mu_N)]  - F(\mu)| \leq \frac{2L}N
		\,.
	\end{equation}
	From~\eqref{eq fmu2} with i.i.d $(X^\ast_i)_{i=1}^N$ such that $X^\ast_i \sim m^\ast$, $i=1,\ldots, N$ we get that   
	\[
	\left| \mathbb E \left[ F\left(\frac1N \sum_{i=1}^N \delta_{X^\ast_i}\right)\right]  -  F(m^\ast) \right| \leq \frac{2L}N \,.
	\]
	Let $(X^\ast_i)_{i=1}^N$ be i.i.d. such that $X^\ast_i \sim m^\ast$, $i=1,\ldots, N$.
	Note that
	\[F(m^\ast) \leq \inf_{(x_i)_{i=1}^N\subset \mathbb R^d} F\left(\frac1N \sum_{i=1}^N \delta_{x_i}\right) \leq \mathbb E \left[ F\left( \frac1N \sum_{i=1}^N \delta_{X^\ast_i} \right)\right]\,.
	\] 
	From this and~\eqref{eq fmu2} we then obtain
	\[
	0\le \inf_{(x_i)_{i=1}^N\subset \mathbb R^d} F\left(\frac1N \sum_{i=1}^N \delta_{x_i}\right) -  F(m^\ast)  \leq \frac{2L}N \,.
	\]
	
\end{proof}

\no In the rest of the section, we shall discuss the first order condition for the minimizer of the function $V^\si$.  
We first show an elementary lemma for convex functions on $\cP(\dbR^d)$.

\begin{lem}
	Under Assumption \ref{assum:energy}, given $m,m'\in \cP(\dbR^d)$, we have
	\bea\label{eq:convexfirstorder}
	F(m') - F(m) \ge \int_{\dbR^d} \frac{\d F}{\d m}(m,x)
	(m'-m)(\mathrm{d}x).
	\eea
\end{lem}
\begin{proof}
	Define $m^\eps: = (1-\eps)m + \eps m'$. Since $F$ is convex, we have 
	\beaa
	\eps\Big(F(m') - F(m)\Big) \ge F(m^\e) - F(m)=\int_{0}^{\eps}\int_{\mathbb{R}^{d}}\frac{\delta F}{\delta m}(m^s,x) (m' -m)(\mathrm{d}x)\mathrm{d}s
	\eeaa
	Since $ \frac{\delta F}{\delta m}$ is bounded and continuous, we obtain \eqref{eq:convexfirstorder} by the dominant convergence theorem. 
\end{proof}

\ms

\begin{proof}[Proof of Proposition \ref{prop:firstorder}]
	~\\
	{\it Step 1}.\q We first prove the existence of minimizer. Clearly there exists $ \bar{m}\in\mathcal{P}(\mathbb{R}^d) $ such that $ V^\sigma(\bar{m}) <+\infty $. Denote 
	\beaa
	\mathcal{S}:=\left\{ m:\frac{\sigma^2}{2}H(m) \leq V^\sigma(\bar{m}) - \inf_{m'\in\mathcal{P}(\mathbb{R}^d)} F(m') \right\}.
	\eeaa
	As a sublevel set of the relative entropy $H$, $\cS$ is weakly compact, see e.g. \cite[Lemma 1.4.3]{DE97}.  Together with the weak lower semi-continuity of $ V^\sigma $, the minimum of $V^\si$ on $ \mathcal{S} $ is attained. Notice that for all $ m\notin\mathcal{S} $, we have $ V^\sigma(m)\geq V^\sigma(\bar{m}) $, so  the minimum of $ V^\sigma $ on $ \mathcal{S} $ coincides with the global minimum. Further, since $V^\si$ is strictly convex, the minimizer is unique. 
	Moreover, given $ m^* =\argmin\limits_{m\in \cP(\dbR^d)}  V^\sigma(m) $,  we know $ m^*\in\mathcal{S} $, and thus we have  $ H(m^*)<\infty $ as well as $\dbE^{m^*}[U(X)]<\infty$. Therefore, $ m^* \in \cP_2(\dbR^d)$ is absolutely continuous with respect to the Gibbs measure, so also absolutely continuous with respect to the Lebesgue measure.
	
	\ms

	\no{\it Step 2}.\q \underline{Sufficient condition:} Let $m^*\in\cP_2(\dbR^d)$ satisfy \eqref{InvariantSet}, in particular, $m^*$ is equivalent to the Lebesgue measure. 
	Let $ m \in \cP(\dbR^d) $ such that $ m $ is absolutely continuous with respect to the Lebesgue measure (otherwise $ V^\sigma(m)=+\infty $).
	Let 
	\[ 
	f:=\frac{\mathrm{d}m}{\mathrm{d}m^*} 
	\]
	be the Radon-Nikodym derivative. Let $m^\eps := (1-\eps)m^* + \eps m = (1+\eps (f-1))m^* $ for $\eps>0$. For the simplicity of the notations, denote $ m^\eps(x) $ and $ m^*(x) $ the respective density function of $ m^\eps $ and $ m^* $ with respect to Lebesgue measure.
	Recall the function $h$ in \eqref{eq:funh} and note that $h(y) \ge y-1$ for all $y\in \dbR^+$. 
	Using \eqref{eq:convexfirstorder}, we obtain
	\[
	\frac{F(m^\eps) - F(m^\ast)}{\eps} \geq \frac1\eps \int_{\mathbb R^d} \frac{\delta F}{\delta m}(m^\ast,\cdot) (m^\eps - m^\ast)\,\mathrm{d}x
	= \int_{\mathbb R^d} \frac{\delta F}{\delta m}(m^\ast,\cdot) \, (f-1)m^\ast\,\mathrm{d}x\,.
	\]
	Moreover 
	\[
	\begin{split}
		\frac{\sigma^2}{2\eps}\bigg(H(m^\eps) - H(m^\ast)\bigg) 
		& = \frac{\sigma^2}{2\eps}\int_{\mathbb R^d} \bigg(m^\eps \log \frac{m^\eps}{g} - m^\ast \log \frac{m^\ast}{g} \bigg)\,\mathrm{d}x\\
		& = \frac{\sigma^2}{2\eps}\int_{\mathbb R^d} (m^\eps - m^\ast)\log\frac{m^\ast}{g}\,\mathrm{d}x + \frac{\sigma^2}{2\eps}\int_{\mathbb R^d} m^\eps \bigg(\log \frac{m^\eps}{g} - \log \frac{m^\ast}{g} \bigg)\,\mathrm{d}x\\
		& = \frac{\sigma^2}{2}\int_{\mathbb R^d} (f-1)m^\ast \log\frac{m^\ast}{g}\,\mathrm{d}x + \frac{\sigma^2}{2\eps}\int_{\mathbb R^d} m^\eps \log \frac{m^\eps}{m^\ast} \,\mathrm{d}x\\
		& = \frac{\sigma^2}{2}\int_{\mathbb R^d} (f-1)m^\ast \big(\log m^\ast + U\big)\,\mathrm{d}x + \frac{\sigma^2}{2\eps}\int_{\mathbb R^d} h(1+\eps(f-1))m^\ast \,\mathrm{d}x\\
		& \geq \frac{\sigma^2}{2}\int_{\mathbb R^d} (f-1)m^\ast \big(\log m^\ast + U\big)\,\mathrm{d}x + \frac{\sigma^2}{2}\int_{\mathbb R^d} (f-1)m^\ast \,\mathrm{d}x\\
		& = \frac{\sigma^2}{2}\int_{\mathbb R^d} (f-1)m^\ast \big(\log m^\ast + U\big)\,\mathrm{d}x
		\\
	\end{split}
	\]
	since $\int_{\mathbb R^d} (f-1)m^\ast \, \mathrm{d}x = \int_{\mathbb R^d} (m-m^\ast) \, \mathrm{d}x = 0$.
	Hence 
	\[
	\frac{V^\sigma(m^\eps)-V^\sigma(m^*)}{\eps} 
	\geq \int_{\mathbb R^d} \bigg(\frac{\delta F}{\delta m}(m^\ast,\cdot) + \frac{\sigma^2}{2} \log m^\ast + \frac{\sigma^2}{2} U \bigg) \, (f-1)m^\ast\,\mathrm{d}x = 0.
	\]

	\ms
	\no{\it Step 3}.\q   \underline{Necessary condition:} Let $ m^* $ be the minimizer of $ V^\sigma $. Let $ m $ a probability measure such that $ H(m)<\infty $, in particular $ m $ is also absolutely continuous with respect to Lebesgue measure $ \ell $. As above, denote $ m(x) $ and $ m^*(x) $ the respective density function of $ m $ and $ m^* $ with respect to Lebesgue measure and we have
	\beaa
	&&\frac{V^\sigma(m^\eps)-V^\sigma(m^*)}{\eps}\\ 
	&= &\frac{1}{\eps}\int_{0}^{\eps}\int_{\mathbb{R}^{d}}\frac{\delta F}{\delta m}(m^s,x)(m(x)-m^*(x))\mathrm{d}x\mathrm{d}s \\
	&&\q\q\q\q\q +  \frac{\sigma^2}{2\eps}\int_{\mathbb{R}^{d}}\left(h(m^\eps(x))-h(m^*(x))-\log(g(x))(m^\eps(x)-m^*(x))\right)\mathrm{d}x.
	\eeaa
	Since $ h $ is convex, we note that for all $ \eps\in(0,1) $
	\beaa
	\frac1{\eps}(h(m^\eps(x))-h(m^*(x))-\log(g(x))(m^\eps(x)-m^*(x)))\leq m(x)\log\left( \frac{m(x)}{g(x)}\right)-m^*(x)\log\left( \frac{m^*(x)}{g(x)}\right).
	\eeaa
	Since $ H(m) $ and $ H(m^*) $ are both finite, the right hand side of the above inequality is integrable. Therefore by Fatou's Lemma we obtain
	
	\begin{equation}\label{eq:nc}
		0 \leq \ol{\lim\limits_{\eps\to 0}}\frac{V^\sigma(m^\eps)-V^\sigma(m^*)}{\eps}\leq \int_{\dbR^d}\left(\frac{\delta F}{\delta m}(m^*,x) + \frac{\sigma^2}{2}\log(m^*(x)) +\frac{\sigma^2}{2}U(x) \right)(m(x)-m^*(x))\mathrm{d}x.
	\end{equation}
	Since $ m $ is arbitrary, we first obtain 
	\begin{equation*}
		\frac{\delta F}{\delta m}(m^*,\cd) + \frac{\sigma^2}{2}\log(m^*) + \frac{\sigma^2}{2}U~\text{ is a constant, }m^*-a.s.
	\end{equation*}
	Now suppose that $ m^* $ is not equivalent to Lebesgue measure. There exists a set $ \cK\subset\dbR^d $ such that $ m^*(\cK)=0 $ and $ \ell(\cK)>0 $. It follows from \eqref{eq:nc} that $ 0\leq C-\int_{\cK}\infty\mathrm{d}m $. Since we may choose $ m $ having positive mass on $ \cK $, it is a contradiction. Therefore, $ m^* $ is equivalent to Lebesgue measure and we have
	\begin{equation*}
		\frac{\delta F}{\delta m}(m^*,\cd) + \frac{\sigma^2}{2}\log(m^*) + \frac{\sigma^2}{2}U~\text{ is a constant, }\ell-a.s.
	\end{equation*}
	
\end{proof}

%

\section{Mean Field Langevin Equations}
\label{sec mf lang}

Recall that
\begin{equation*}
	b(x, m) := D_mF(m,x) + \frac{\sigma^2}{2} \nabla U(x).
\end{equation*}
Due to Assumption \ref{A1} and \ref{assum:U}, the function  $b$ is of linear growth.

\begin{lem}\label{lem:uniL2}
	Under Assumption \ref{assum:U}  and  \ref{A1}, let $X$ be the strong solution to \eqref{Langevin}. 
	If $m_0\in \cP_p(\dbR^d)$ for some $p\ge 2$, we have
	\bea\label{estimate:supP}
	\dbE\Big[ \sup_{t\le T} |X_t |^p \Big] \le C,\q\q\mbox{for some $C$ depending on $p,\si, T$}.
	\eea
	If $m_0\in \cP_p(\dbR^d)$ for some $p\ge 2$, we have
	\bea\label{estimate:compact}
	\sup_{t\in\dbR^+}\dbE\Big[  |X_t |^p \Big] \le C,\q\q\mbox{for some $C$ depending on $p,\si$}.
	\eea
	In particular, if $m_0\in \cup_{p>2}\cP_p(\dbR^d)$, then $(m_t)_{t\in \dbR^+}$ belong to a $\cW_2$-compact subset of $\cP_2(\dbR^d)$. 
\end{lem}
\begin{proof}
	Since $b$ is of linear growth, we have
	\beaa
	|X_t| \le |X_0| + \int_0^t C(1+ |X_t|) \mathrm{d}t + \left| \si W_t \right|.
	\eeaa
	Therefore, 
	\beaa
	\sup_{t\le s}|X_t|^p \le C\left( |X_0|^p + 1+\int_0^ s \sup_{t\le r}| X_t|^p \mathrm{d}r+ \sup_{t\le s}\left| \si W_t \right|^p  \right).
	\eeaa
	Note that $\dbE\left[ \sup_{t\le s}\left| \si W_t \right|^p\right] \le Cs^{p/2}$. Then \eqref{estimate:supP} follows from the Gronwall inequality. 
	
	For the second estimate, we apply the It\^o formula and obtain
	\beaa
	\mathrm{d} |X_t|^p = |X_t |^{p-2} \Big( -p  X_t\cd b(X_t, m_t) + \frac{p(p-1)}{2}\si^2 \Big)\mathrm{d}t  + p\si |X_t|^{p-2} X_t  \cd \mathrm{d}W_t.
	\eeaa
	Since $D_m F$ is bounded and $\nabla U(x)\cd x\ge C|x|^2 +C'$, we have
	\beaa
	\mathrm{d}|X_t|^p &\le& |X_t |^{p-2} \Big( C''|X_t| -  \frac{p\si^2}{2}(C|X_t|^2+C') + \frac{p(p-1)}{2}\si^2 \Big)\mathrm{d}t  + p\si |X_t|^{p-2} X_t  \cd \mathrm{d}W_t\\
	&\le &|X_t |^{p-2} \Big( C -  \e |X_t|^2 + \frac{p(p-1)}{2}\si^2 \Big)\mathrm{d}t  + p\si |X_t|^{p-2} X_t  \cd \mathrm{d}W_t, ~\mbox{for some  $0<\e< \frac{p\si^2C}{2}$.}
	\eeaa
	The last inequality is due to the Young inequality. Again by the It\^o formula we have
	\bea\label{eq:X^pestimate}
	\mathrm{d}\big(e^{\e t} |X_t|^p\big)  ~\le~ e^{\e t}\left( |X_t |^{p-2} \Big( C  + \frac{p(p-1)}{2}\si^2 \Big)\mathrm{d}t  + p\si |X_t|^{p-2} X_t  \cd \mathrm{d}W_t \right)
	\eea
	Further, define the stopping time $\t_m:=\inf\{t\ge 0: |X_t|\ge m\}$. By taking expectation on both sides of \eqref{eq:X^pestimate}, we have
	\bea\label{eq:itoLp}
	\dbE[e^{\e(\t_m\we t) }|X_{\t_m\we t}|^p] \le  \dbE[|X_0|^p] + \dbE \left[\int_0^{\t_m\we t} e^{\e s} |X_s |^{p-2} \Big( C + \frac{p(p-1)}{2}\si^2 \Big)\mathrm{d}s\right].
	\eea
	In the case $p=2$,  it follows from the Fatou lemma and the monotone convergence theorem that
	\beaa
	\dbE[|X_{ t}|^2] \le e^{-\e t}  \dbE[|X_0|^2]  +  \int_0^{ t} e^{\e (s-t)} \Big( C + \si^2 \Big)\mathrm{d}s \le C\left(e^{-\e t}+ \e^{-1}(1 -e^{-\e t})\right),
	\eeaa
	and thus $\sup_{t\in \dbR^+}\dbE[|X_{ t}|^2]<\infty$. For  $p> 2$, we again obtain from \eqref{eq:itoLp} that
	\beaa
	\dbE[|X_{t}|^p] \le  e^{-\e t}  \dbE[|X_0|^p] + \int_0^{t} e^{\e (s-t)} \dbE[ |X_s |^{p-2} ]\Big( C + \frac{p(p-1)}{2}\si^2 \Big)\mathrm{d}s.
	\eeaa
	Then \eqref{estimate:compact} follows from induction.
\end{proof}

\begin{prop}\label{prop:smoothsol}
	Let Assumption \ref{assum:U} and \ref{A1} hold true and assume $m_0\in \cP_p(\dbR^d)$ for some $p\ge 2$. The  marginal law $ m $ of the solution $X$ to \eqref{Langevin} is a weak solution to Fokker-Planck equation:
	\begin{equation}\label{Fokker-Planck}
		\partial_tm = \nabla\cdot\left(b(x,m)m + \frac{\sigma^2}{2}\nabla m\right),
	\end{equation}
	in the sense that for all $C^\infty$-function $\phi: \dbR^d\rightarrow \dbR$ such that $\phi, \nabla \phi, \nabla^2 \phi$ decay to $0$ at infinity,  we have
	\beaa
	\int_{\dbR^d} \phi(x) \big(m_t -m_s\big)(dx) =\int_s^t \int_{\dbR^d} \Big(- \nabla\phi(x) b(x, m_u) m_u(dx) + \frac{\si^2}{2} \D \phi(x) m_u(dx)\Big) du.
	\eeaa
	Moreover, the mapping $t\mapsto m_t$ is weakly continuous on $[0,+\infty)$, the joint density function $(t,x)\mapsto m(t,x)$ exists and   $m\in C^{1,\infty} \left((0,\infty)\times \dbR^d, \dbR \right)$.
	In particular, a stationary solution to the Fokker-Planck equation \eqref{Fokker-Planck} is an invariant measure to \eqref{Langevin}.
\end{prop}
\begin{proof}
	By applying the It\^o formula on $\phi(t,X_t)$, we can verify that $m$ is a weak solution to \eqref{Fokker-Planck}. Next, define $\tilde b(x,t):= b(x, m_t)$. Obviously, $m$ can be regarded as a weak solution to the linear PDE:
	\bea\label{eq:linearPDE}
	\partial_tm = \nabla\cdot\left(\tilde b m + \frac{\sigma^2}{2}\nabla m\right).
	\eea
	Then the regularity result follows from a standard argument through $L^p_{loc}$-estimate. For details, we refer the readers to the seminal paper \cite[p.14-p.15]{JKO98} or the classic book \cite[Chapter IV]{LSU68}. 
	
	Let $m^*$ be a stationary solution to \eqref{Fokker-Planck}, and $X$ be the strong solution to the SDE:
	\beaa
	dX_t = -b(X_t, m^*)dt  + \si dW_t.
	\eeaa
	It is easy to verify that given ${\rm Law}(X_0) = m^*$ we have ${\rm Law}(X_t) = m^*$ for all $t\ge 0$. Therefore $X$ is the solution to mean-field Langevin equation \eqref{Langevin} and $m^*$ is an invariant measure. 
\end{proof}


\section{Convergence to the Invariant Measure}
\label{sec conv}

\no Now we are going to show that under mild conditions, the flow of marginal law $(m_t)_{t\in \dbR^+}$ converges toward the invariant measure which coincides with the minimizer of $V^\si$.

\begin{lem}\label{finiteEntropy}
	Suppose Assumption \ref{assum:U} and \ref{A1} hold true and $m_0\in \cP_2(\dbR^d)$. Let $ m $ be the law of  the solution to the mean field Langevin equation \eqref{Langevin}. Denote by $ \dbP_{\sigma,w} $ the scaled Wiener measure\footnote{Under the scaled Wiener measure $ \dbP_{\sigma,w} $, if we denote $ X $ as the canonical process, $ \frac{X}{\sigma} $ is a standard Brownian motion.} with initial distribution $ m_0 $. Then,
	\begin{enumerate}[i)] 
		\item For any $T>0$, $ \dbP_{\sigma,w} $ is equivalent to $ m $ on $\cF_T$, where $\{\cF_t\}$ is the filtration generated by $X$, and the relative entropy 
		\bea\label{eq:r.e.wiener}
		\dbE^{m}\left[\log \Big(\frac{\mathrm{d}m}{\mathrm{d}\dbP_{\sigma,w}}\Big|_{\cF_T} \Big)\right]<\infty.
		\eea
		
		\item  For all $ t>0$,  the marginal law $m_t$ admits density such that $m_t>0$ and $H(m_t)<\infty$.
	\end{enumerate}
\end{lem}

\begin{proof}
	
	\no{\it i)} \q We shall prove in the Appendix in Lemma \ref{measureEquivalent} that due to the linear growth in $x$ of the drift $b$,  $ \dbP_{\sigma,w} $ is equivalent to $ m $. Also by the linear growth of coefficient, we have
	\beaa
	\dbE^{m}\left[\log \Big(\frac{\mathrm{d}m}{\mathrm{d}\dbP_{\sigma,w}}\Big|_{\cF_T} \Big)\right] 
	= \dbE^{m}\left[ \frac{1}{\sigma^2}\int_0^T |b(X_t, m_t)|^2  \mathrm{d}t\right]
	\le C\dbE^{m}\left[ 1+ \sup_{t\le T}|X_t|^2\right]<\infty.
	\eeaa
	The last inequality is due to Lemma \ref{lem:uniL2}.
	
	\no{\it ii) }\q Since $ \dbP_{\sigma,w} $ is equivalent to $ m $, we have $ m_t>0 $. Denote $ f_{\sigma,t} $ the density function of the marginal law of a standard Brownian motion multiplied by $ \sigma $ with initial distribution $ m_0 $. It follows from the conditional Jensen inequality that for all $t\in [0,T]$
	\bea\label{eq:pushforward}
	\int_{\mathbb{R}^{d}}m_t\log\left(\frac{m_t(x)}{f_{\sigma,t}(x)}\right)\mathrm{d}x \leq \dbE^{m}\left[\log \Big(\frac{\mathrm{d}m}{\mathrm{d}\dbP_{\sigma,w}}\Big|_{\cF_T} \Big)\right] < +\infty. 
	\eea
	Further, by the fact $f_{\sigma,t} (x) \le \frac{1}{(2\pi t)^{d/2}\sigma}$, we have
	\beaa
	\int_{\mathbb{R}^{d}}m_t(x)\log(f_{\sigma,t}(x))\mathrm{d}x \le -\frac{d}{2} \log(2\pi t \sigma^2).
	\eeaa
	Finally, note that
	\beaa
	- \int_{\mathbb{R}^{d}}m_t(x)\log\big(g(x)\big)\mathrm{d}x = \int_{\mathbb{R}^{d}}m_t(x)U(x)\mathrm{d}x 
	\le C \int_{\mathbb{R}^{d}}m_t(x)|x|^2\mathrm{d}x <\infty.
	\eeaa
	Together with \eqref{eq:pushforward}, we have $H(m_t )<\infty$.
\end{proof}


\no Next, we introduce an interesting result of \cite[Theorem 3.10 and Remark 4.13]{Follmer}.
\begin{lem}\label{Fisher}
	Let $ m $ be a measure equivalent to the scaled Wiener measure $\dbP_{\si,w}$ such that the relative entropy is finite as in \eqref{eq:r.e.wiener}.  Then,
	\begin{enumerate}[i)]
		\item	for any $ 0<t<T $ we have $\int_{t}^{T}\int_{\mathbb{R}^{d}}\left|\nabla \log (m_s)\right|^2 m_s\mathrm{d}x\mathrm{d}s<+\infty$.
		\item   given $t \ge t_0>0$ such that the Dol\'eans-Dade exponential $\cE^b(X):= e^{-\int_{t-t_0}^{t} \frac{b_s}{\si^2} \mathrm{d}X_s -\int_{t-t_0}^{t} \frac12 |\frac{b_s}{\si}|^2 ds}$ \footnote{Again, we slightly abuse the notation, using $X$ to denote the canonical process of the Wiener space.} is conditionally  $\dbL^2$-differentiable on the interval $[t-t_0, t]$\footnote{Denote by $\dbP^{t-t_0, x_0}_{\si,w}$  the conditional probability of $\dbP_{\si,w}$ given $X_{t-t_0}=x_0$.  $\cE^b(X)$ is conditionally  $\dbL^2$-differentiable on the interval $[t-t_0, t]$, if there exists an absolutely continuous process  $D\cE^b:=\int_{t-t_0}^\cd D\cE^b_s ds$ with $D\cE^b_s\in \dbL^2(\dbP^{t-t_0, x_0}_{\si,w})$  for all $x_0\in \dbR^d$ such that for any $ h := \int_{t-t_0}^{\cdot}\dot{h}_s\mathrm{d}s $ with bounded predictable $ \dot{h} $, we have
			\bea\label{EbL2diff}
			\lim_{\e\rightarrow 0} \Big| \frac{\cE^b(X+\e h) -\cE^b(X)}{\e} - \langle D\cE^b(X), h \rangle\Big| =0, \q\mbox{in $\dbL^2(\dbP^{t-t_0, x_0}_{\si,w})$ for all $x_0\in \dbR^d$,}
			\eea
			where   $
			\langle D\cE^b(X), h \rangle = \int_{t-t_0}^{t} \dot{h}_{s} D\cE^b_s(X) d s$.},
		we have	
		\bea\label{eq:explicitefisher}
		\nabla \log\big(m_t(x)\big) = -\frac{1}{t_0} \dbE\Big[\int_{0}^{t_0} \big(1+ s \nabla b(X_{t-t_0+s}, m_{t-t_0+s})\big) \mathrm{d}W^{t-t_0}_s\Big| X_t =x\Big],
		\eea
	\end{enumerate}
	where $ W^{t-t_0}_s := W_{t-t_0+s} - W_{t-t_0}$ and $W$ is the Brownian motion in \eqref{Langevin}.
	
\end{lem}

\no We shall prove in the Appendix, Lemma \ref{lem:L2diff}, that under Assumption \ref{assum:U} and \ref{A1}, $\cE^b$ is conditionally  $\dbL^2$-differentiable on $[t-t_0, t]$ for all $t\ge t_0>0$.

The estimate (i) leads to some other integrability results.

\begin{lem}
	Suppose Assumption \ref{assum:U} and \ref{A1} hold true and $m_0\in \cP_2(\dbR^d)$. We have
	\beaa
	\int_{t}^{T}\int_{\mathbb{R}^{d}} |\nabla m_t(x) |\mathrm{d}x\mathrm{d}t<\infty
	\q\mbox{and}\q
	\int_{t}^{T}\int_{\mathbb{R}^{d}} |x\cd \nabla m_t(x) |\mathrm{d}x\mathrm{d}t<\infty.
	\eeaa
\end{lem}

\begin{proof}
	By the Young inequality, we have
	\begin{equation*}
		|\nabla m_t| \leq m_t + \left|\frac{\nabla m_t}{m_t}\right|^2m_t\q
		\mbox{and}\q
		|x\cdot \nabla m_t| \leq x^2m_t + \left|\frac{\nabla m_t}{m_t}\right|^2m_t.
	\end{equation*}
	Since all terms on the right hand sides are integrable, due to Lemma \ref{Fisher},  so are $ \nabla m $ and $x\cd\nabla m$.
\end{proof}

\no Based on the previous integrability results, the next lemma follows from  the integration by part. 
\begin{lem}\label{green's formula}
	Let $m_0\in \cP_2(\dbR^d)$. Under Assumption \ref{assum:U} and \ref{A1}  we have   for {\it Leb}-a.s. $t$ that
	\beaa
	&\int_{\mathbb{R}^{d}}\Tr(\nabla D_mF(m_t,x)) m_t\mathrm{d}x = -\int_{\mathbb{R}^{d}}D_mF(m_t,x)\cd\nabla m_t\mathrm{d}x,&\\
	&
	\mbox{and}\q 
	\int_{\mathbb{R}^{d}} \D U(x) m_t(x)\mathrm{d}x =- \int_{\mathbb{R}^{d}} \nabla U(x)\cdot\nabla m_t(x)\mathrm{d}x.&
	\eeaa
\end{lem}

Again using the estimate (i) in Lemma \ref{Fisher}, together with Theorem 2.1 of Haussmann and Pardoux \cite{HP86}, we directly obtain the following result concerning the time reverse process $\widetilde X_t : = X_{T-t}$ for a given $T>0$ and $t\le T$.

\begin{lem}
	Under Assumption \ref{assum:U} and \ref{A1} , there exists a Brownian motion $\widetilde W_t$ such that $(\widetilde X, \widetilde W)$ is a weak solution to the SDE:
	\beaa
	d\widetilde X_t = \left( b(\widetilde X_t, m_{T-t}) + \si^2 \nabla \log m_{T-t} (\widetilde X_t) \right) dt +\si d \widetilde W_t.
	\eeaa
\end{lem}

\ms

\begin{proof}[Proof of Theorem \ref{thm:Vdecrease}]
	By the It\^o formula and the Fokker-Plank equation \eqref{Fokker-Planck}, we have
	\beaa
	d \log m_{T-t}(\widetilde X_t) &=& \Big(- \frac{\pa_t m_{T-t}}{m_{T-t}}(\widetilde X_t) + \nabla \log(m_{T-t}(\widetilde X_t))
	\cd \big(  b(\widetilde X_t, m_{T-t}) + \si^2 \nabla \log m_{T-t} (\widetilde X_t) \big) \\
	&& \q\q\q\q +\frac12\si^2 \D \log (m_{T-t}(\widetilde X_t))	\Big) dt + \nabla \log(m_{T-t}(\widetilde X_t)) \cd d \widetilde W_t\\
	& = & \Big(\frac{\si^2}{2}\left|\frac{\nabla m_{T-t}}{m_{T-t}}(\widetilde X_t) \right|^2 - \nabla\cdot b(\widetilde X_t, m_{T-t})  \Big)d t + \nabla \log(m_{T-t}(\widetilde X_t)) \cd d \widetilde W_t.
	\eeaa 
	Next by  the It\^o formula we obtain
	\beaa
	\mathrm{d} U(X_t) =  \left( - \nabla U(X_t)\cdot b(X_t,m_t) + \frac{\sigma^2}{2}\Delta U(X_t) \right)\mathrm{d}t + \nabla U(X_t) dW_t.
	\eeaa
	Note that $d H(m_t) = d \dbE[ \log m_{t}(\widetilde X_{T-t}) + U(X_t)] $. Therefore, it follows from Lemma \ref{green's formula} that
	\bea		\label{FreeEnergyEq02}
	&&\mathrm{d}H(m_t) \notag\\
	&=&\dbE \left[ - \frac{\si^2}{2}\left|\frac{\nabla m_t}{m_t}(X_t) \right|^2 - b(X_t,m_t) \cd\frac{\nabla m_t}{m_t}(X_t) - \nabla U(X_t)\cdot b(X_t,m_t)  - \frac{\sigma^2}{2} \nabla U(X_t)\cd \frac{\nabla m_t}{m_t}(X_t)\right]\mathrm{d}t \notag \\
	& = & \dbE\left[ - \frac{\si^2}{2}\left|\frac{\nabla m_t}{m_t}(X_t) + \nabla U(X_t)\right|^2 - D_m F(X_t,m_t) \cd\left(\frac{\nabla m_t}{m_t}(X_t) + \nabla U(X_t)\right) \right]\mathrm{d}t \notag\\
	&=& \int_{\dbR^d} \left( - \frac{\si^2}{2}\left| \frac{\nabla m_t}{m_t}+ \nabla U(x)\right|^2 -D_m F(m_t,x) \cd \left(\frac{\nabla m_t}{m_t}+ \nabla U(x) \right) \right)m_t(x)\mathrm{d}x
	\eea
	Further, by the It\^o-type formula given by \cite[Theorem 4.14]{Carmona+Rene} and  Lemma \ref{green's formula}, we have
	\bea
	\mathrm{d}F(m_t) 
	&=& \int_{\mathbb{R}^d} \left(-|D_mF(m_t,x)|^2 - \frac{\sigma^2}{2}D_mF(m_t,x)\cd \nabla U(x) + \frac{\sigma^2}{2}\Tr(\nabla D_mF(m_t,x))\right) m_t\mathrm{d}x\mathrm{d}t \nonumber \\
	&=& \int_{\mathbb{R}^d} \left(-|D_mF(m_t,x)|^2 - \frac{\sigma^2}{2}D_mF(m_t,x)\cd\Big( \nabla U(x) +\frac{\nabla m_t}{m_t}\Big)\right) m_t\mathrm{d}x\mathrm{d}t. \label{FreeEnergyEq01}
	\eea		
	Finally, summing up the equation \eqref{FreeEnergyEq02} and \eqref{FreeEnergyEq01}, we obtain \eqref{eq: measureFlow}.
	%
\end{proof}


In order to prove there exists an invariant measure of \eqref{Langevin} equal to the minimizer of $V^\si$, we shall apply Lasalle's invariance principle.
Now we simply recall it in our context. Let $(m_t)_{t\in \dbR^+}$ be the flow of marginal laws of the solution to \eqref{Langevin}, given an initial law $m_0$. Define a dynamic system $S(t)[m_0]:=m_t$. We shall consider the so-called $w$-limit set:
\beaa
w(m_0):= \left\{\mu\in \cP_2(\dbR^d):~~\mbox{there exist}~t_n\rightarrow\infty~\mbox{such that}~~ \cW_2\Big(S(t_n)[m_0],\mu\Big)\rightarrow 0\right\}
\eeaa

\begin{prop}\label{prop:invprin}[Invariance Principle]
	Let Assumption \ref{A1} hold true and assume that $m_0\in \cup_{p>2}\cP_p(\dbR^d)$. Then the set $w(m_0)$ is nonempty, compact and  invariant, that is, 
	\begin{enumerate}[i)]
		\item for any $\mu\in w(m_0)$, we have $S(t)[\mu]\in w(m_0)$ for all $t\in \dbR^+$.
		\item for any $\mu\in w(m_0)$ and all $t\in \dbR^+$, there exits $\mu'\in w(m_0)$ such that $S(t)[\mu'] =\mu$.
	\end{enumerate}
\end{prop}
\begin{proof}
	Under the upholding assumptions, it follows from Proposition \ref{prop:wellpose} that $S(t)$ is continuous with respect to the $\cW_2$-topology. By Lemma \ref{lem:uniL2}, we have \eqref{estimate:compact} with $p>2$, and thus $\big(S(t)[m_0]\big)_{t\in \dbR^+} = (m_t)_{t\in \dbR^+}$ live in a $\cW_2$-compact subset of $\cP_2(\dbR^d)$. 
	The desired result follows from the invariance principle, see e.g. \cite[Theorem 4.3.3]{Henry81}. 
	In order to keep the paper self-contained, we state the proof as follows. 
	
	First, for any $t\ge 0$, $(m_s)_{s\geq t}$ is relatively compact, hence $\overline{(m_s)_{s\geq t}}$ is compact. Since the arbitrary intersection of closed sets is closed, the set    
	\[
	w(m_0) = \bigcap_{t \geq 0}\overline{(m_s)_{s\geq t}} 
	\] 
	is compact.
	
	Next, let $ \mu\in w(m_0) $, by definition we know that there exists a sequence $ (t_N)_{N>0} $ such that $ S(t_N)[m_0]\to\mu $. Let $ t\in\dbR^+ $, by the continuity of $S(t):\cP_2(\dbR^d)\rightarrow \cP_2(\dbR^d)$, we have $ S(t+t_N)[m_0]\to S(t)[\mu] $ and therefore $ S(t)[\mu]\in w(m_0) $.
	
	Finally, for the second point, let $ t\in\dbR^+ $ and consider the sequence $ \big(S(t_N - t)[m_0]\big)_N $. Since $(m_t)_{t\in \dbR^+}$ live in a $\cW_2$-compact subset of $\cP_2(\dbR^d)$, there exists a subsequence $ (t_{N'}) $ and $ \mu'\in w(m_0) $ such that $ S(t_{N'}-t)[m_0]\to\mu' $. Again, by the continuity of $S(t)$, we have $ S(t)[\mu'] = \lim_{N'\rightarrow\infty} S(t_{N'}-t+t)m_0 = \mu $.
\end{proof}
\ms

\begin{proof}[Proof of Theorem \ref{thm:convergence}]
	\no{\it Step 1.}\q We first prove that $ m^*\in w(m_0) $. Since $w(m_0)$ is compact, there exists $\tilde m\in \argmin\limits_{m\in w(m_0)} V^\si(m) $. By Proposition \ref{prop:invprin}, for $t>0$ there exists  a probability measure $\mu\in w(m_0)$ such that $S(t)[\mu] = \tilde m$. By Theorem \ref{thm:Vdecrease},  for any $s>0$ we have
	\beaa
	V^\si\big( S(t+s) [\mu] \big) \le V^\si(\tilde m).
	\eeaa
	Since $w(m_0)$ is invariant, $S(t+s) [\mu] \in w(m_0)$ and thus $V^\si\big( S(t+s) [\mu] \big) = V^\si(\tilde m)$. Again by Theorem \ref{thm:Vdecrease}, we obtain
	\beaa
	0 = \frac{d V^\si \big(S(t) [\mu] \big)}{dt}  = - \int_{\mathbb{R}^{d}}\left| D_mF(\tilde m,x) + \frac{\sigma^2}{2}\frac{\nabla \tilde m}{\tilde m}(x) + \frac{\sigma^2}{2}\nabla U(x)\right|^2 \tilde m(x)\mathrm{d}x.
	\eeaa
	Since $\tilde m=S(t)[\mu] $ is equivalent to the Lebesgue measure (Proposition \ref{finiteEntropy}), we have
	\bea\label{eq:tildemfirstorder}
	D_mF(\tilde m,\cd) + \frac{\sigma^2}{2}\frac{\nabla \tilde m}{\tilde m} + \frac{\sigma^2}{2}\nabla U =0.
	\eea
	The probability measure $\tilde m$ is an invariant measure of \eqref{Langevin}, because it is a stationary solution to the Fokker-Planck equation \eqref{Fokker-Planck}. Meanwhile, by Proposition \ref{prop:firstorder} we have $\tilde m=m^*$. Therefore, $ m^*\in w(m_0) $.
	
	\ms
	\no{\it Step 2.}\q Since  $ m^*\in w(m_0) $, there exists  a subsequence, denoted by $(m_{t_n})_{n\in\dbN}$,  converging to $m^*$. We are going to prove that $V^\si(m^*)=\lim\limits_{n\rightarrow\infty} V^\si(m_{t_n})$. It is enough to prove $\int_{\dbR^d} m^* \log(m^*) \mathrm{d}x=\lim_{n\rightarrow\infty}\int_{\dbR^d} m_{t_n} \log(m_{t_n}) \mathrm{d}x$. By the lower-semicontinuity of entropy, it is sufficient to prove that
	\bea\label{eq:uppersemicont}
	\int_{\dbR^d} m^* \log(m^*) \mathrm{d}x \ge \limsup_{n\rightarrow\infty}\int_{\dbR^d} m_{t_n} \log(m_{t_n}) \mathrm{d}x
	\eea
	By \eqref{eq:tildemfirstorder}, we know that $-\log m^*$ is semi-convex, so we may apply the HWI inequality in \cite[Theorem 3]{OV00}:
	\bea\label{HWI}
	\int_{\dbR^d} m_{t_n} \Big(\log(m_{t_n})- \log(m^*)\Big) \mathrm{d}x
	\le \cW_2(m_{t_n}, m^*) \Big( \sqrt{I_n} + C \cW_2(m_{t_n}, m^*)\Big),
	\eea
	where $I_n$ is the relative Fisher information defined as
	\bea\label{eq:fisherestimate}
	I_n &: =& \dbE\left[ \left| \nabla \log\Big(m_{t_n}(X_{t_n}) \Big)-\nabla \log\Big(m^*(X_{t_n}) \Big)\right|^2\right]\notag\\
	&=& \dbE\left[ \left| \nabla \log\Big(m_{t_n}(X_{t_n}) \Big)+ \frac{2}{\si^2}D_m F(m^*, X_{t_n}) + \nabla U(X_{t_n}) \right|^2\right].
	\eea
	We are going to show that $\sup_n I_n <\infty$. First, since $D_m F$ is bounded and $\nabla U$ is of linear growth, by Lemma \ref{lem:uniL2} we have 
	\bea\label{eq:fisherestimate1}
	\sup_n \dbE\left[ \left|  \frac{2}{\si^2}D_m F(m^*, X_{t_n}) + \nabla U(X_{t_n}) \right|^2\right]<\infty.
	\eea
	Next, since $\nabla b$ is bounded, by Lemma \ref{lem:L2diff} and \eqref{eq:explicitefisher} we have for all $n$
	\bea\label{eq:fisherestimate2}
	\dbE\left[ \left| \nabla \log\Big(m_{t_n}(X_{t_n}) \Big) \right|^2\right] 
	&\le& 
	\inf_{0< s\le t_n} \frac{1}{s^2} \int_0^{s} C(1+r^2) \, \mathrm{d}r \notag\\
	&=&\inf_{0< s\le t_n}C\Big(\frac{1}{s} + \frac{s}{3}\Big)
	\le  \frac{2C}{\sqrt 3}, \q\q \mbox{for $t_n>\sqrt 3$},
	\eea
	where the constant $C$ does not depend on $n$. Combining \eqref{eq:fisherestimate}, \eqref{eq:fisherestimate1} and \eqref{eq:fisherestimate2} we obtain $\sup\limits_n I_n <\infty$. Now the HWI inequality \eqref{HWI} reads
	\beaa
	\int_{\dbR^d} m_{t_n} \Big(\log(m_{t_n})- \log(m^*)\Big) \mathrm{d}x
	\le C\cW_2(m_{t_n}, m^*) \Big( 1 +  \cW_2(m_{t_n}, m^*)\Big).
	\eeaa
	By letting $n\rightarrow\infty$, since $\cW_2(m_{t_n}, m^*)\rightarrow 0$, we obtain \eqref{eq:uppersemicont}.
	
	\ms
	\no{\it Step 3}.\q Finally we prove the convergence of the whole sequence $(m_t)_{t\in \dbR^+}$ towards $ m^* $, by showing that the set $w(m_0)$ is a singleton, namely $w(m_0)=\{m^*\}$. Since $V^\si(m_t)$ is non-increasing in $t$, there is a constant $c:= \lim\limits_{t\rightarrow\infty} V^\si(m_t)$.  Recall that in {\it Step 2} we proved $V^\si(m^*)=\lim\limits_{n\rightarrow\infty} V^\si(m_{t_n})$, so we obtain $c = V^\si(m^*)$. On the other hand, for any $\mu \in w(m_0)$ there is a  
	subsequence $(m_{t'_n})_{n\in\dbN}$ converging to $\mu$ and by the weak lower-semicontinuity of $ V^\sigma $ we have $V^\si(\mu) \leq \liminf\limits_{n\rightarrow \infty} V^\sigma  (m_{t'_n})=c$. Using the fact that  $m^* =  \argmin\limits_{m\in w(m_0)} V^\si(m)$,  we have
	\beaa
	V^\si(\mu) =V^\si(m^*) = c,  \q \mbox{for all $\mu\in w(m_0)$}.
	\eeaa
	Finally by the uniqueness of the minimiser of $ V^\sigma $, we have $w(m_0) = \{m^*\}$.
	
\end{proof}

\section*{Acknowledgement}
The authors would like to thank the anonymous referees for their valuable comments which have helped improve the clarity of the paper.

The third and fourth authors acknowledge the support of The Alan Turing Institute under the Engineering and Physical Sciences Research Council grant EP/N510129/1.

\appendix

\section{Appendix}

\label{sec appendix}

The following result regarding to the change of measure in the Wiener space is classic, see e.g. \cite{Benes}. For readers' convenience, we provide a transparent proof as follow. Our argument is largely inspired by the one in \cite[Lemma 4.1.1]{Bensoussan}.
\begin{lem}\label{measureEquivalent}
	Let a function $(t,x)\mapsto b(t,x)$ be Lipschitz continuous and of linear growth in $x$, and a process $X$ be the strong solution to the SDE:
	\beaa
	\mathrm{d}X_t = b(t,X_t)\mathrm{d}t + \si \mathrm{d}W_t.
	\eeaa
	Define the following Dol\'ean-Dade exponential for all $t\in\dbR^+$
	\begin{equation}
		\rho_t := \exp\left( \frac{1}{\sigma}\int^t_0 b(s,X_s)\mathrm{d}W_s - \frac{1}{2\sigma^2}\int^t_0|b(s,X_s)|^2\mathrm{d}s \right).
	\end{equation}
	Then we have $\dbE[\rho_t] = 1$ and thus $\rho$ is a martingale on any finite horizon.
\end{lem}
\begin{proof}
	First, we shall prove that there exists $ C>0 $ such that for all $ t\in\dbR^+ $, we have
	\begin{equation}\label{eq: square_bound}
		\dbE[\rho_t|X_t|^2]<C.
	\end{equation}
	By It\^o's formula, we have
	\begin{equation*}
		\mathrm{d}|X_t|^2 = (2X_tb(t,X_t) + \sigma^2)\mathrm{d}t +  2X_t\sigma\mathrm{d}W_t,
	\end{equation*}
	and
	\begin{equation*}
		\mathrm{d}(\rho_t|X_t|^2) = \rho_t\Big(4X_tb(t,X_t) + \sigma^2\Big)\mathrm{d}t +  \rho_t\left(\frac1\sigma|X_t|^2b(t,X_t)+2X_t\sigma\right)\mathrm{d}W_t,
	\end{equation*}
	and further
	\beaa
	\mathrm{d}\frac{\rho_t|X_t|^2}{1+\eps\rho_t|X_t|^2} &=& \frac{\rho_t}{(1+\eps\rho_t|X_t|^2)^2}\left(\frac1\sigma|X_t|^2b(t,X_t)+2X_t\sigma\right)\mathrm{d}W_t \\
	&&  + \frac{\rho_t}{(1+\eps\rho_t|X_t|^2)^2}\Big(4X_tb(t,X_t) + \sigma^2\Big)\mathrm{d}t \\
	&&- \frac{\eps\rho^2_t}{(1+\eps\rho_t|X_t|^2)^3}\left|\frac1\sigma|X_t|^2b(t,X_t)+2X_t\sigma\right|^2\mathrm{d}t.
	\eeaa
	Note that the integrand of the stochastic integral on the right hand side above is bounded, so the stochastic integral is actually a real martingale. Therefore, by taking the expectation on both sides and using the fact that $ b$ has linear growth in $x$, we get
	\beaa
	\frac{\mathrm{d}}{\mathrm{d}t}\dbE\left[ \frac{\rho_t|X_t|^2}{1+\eps\rho_t|X_t|^2} \right] &\leq& \dbE\left[ \frac{\rho_t}{(1+\eps\rho_t|X_t|^2)^2}\Big(4X_tb(t,X_t) + \sigma^2\Big) \right] \\
	&\leq& K\dbE\left[ \frac{\rho_t|X_t|^2}{1+\eps\rho_t|X_t|^2}+1 \right].
	\eeaa
	By Gr\"{o}nwall inequality, we get
	\begin{equation*}
		\dbE\left[ \frac{\rho_t|X_t|^2}{1+\eps\rho_t|X_t|^2} \right] \leq C,
	\end{equation*}
	for some constant $ C $ which does not depend on $ \eps $. By Fatou's lemma, we get \eqref{eq: square_bound}. 
	
	Next, by It\^o's formula, we have
	\begin{equation*}
		\mathrm{d}\frac{\rho_t}{1+\eps\rho_t} = \frac{\rho_tb(t,X_t)}{(1+\eps\rho_t)^2}\mathrm{d}W_t - \frac{\eps\rho^2_tb(t,X_t)^2}{(1+\eps\rho_t)^3}\mathrm{d}t.
	\end{equation*}
	By \eqref{eq: square_bound}, the stochastic integral above is a martingale, so taking the expectation on both sides, we get
	\begin{equation*}
		\dbE\left[ \frac{\rho_t}{1+\eps\rho_t} \right] = \frac1{1+\eps} -\int_{0}^{t}  \dbE\left[ \frac{\eps\rho^2_sb(s,X_s)^2}{(1+\eps\rho_s)^3}\right] \mathrm{d}s.
	\end{equation*}
	Due to the linear growth of $ b $, the term inside the expectation  on the right hand side is bounded by $ C\rho_s(|X_s|^2+1) $ for some constant $ C>0 $ independent of $\e$. By the dominated convergence theorem, we get
	\begin{equation*}
		\lim_{\eps\to0}\dbE\left[ \frac{\rho_t}{1+\eps\rho_t} \right] = 1.
	\end{equation*}
	To conclude, one only needs to note that
	$\lim_{\eps\to0}\dbE\left[ \frac{\rho_t}{1+\eps\rho_t} \right] = \dbE[\rho_t]$.
\end{proof}

\ms

\begin{lem}\label{lem:L2diff}
	Under Assumption \ref{assum:U} and \ref{A1}, the exponential martingale $\cE(b)$ is  conditionally $\dbL^2$-differentiable on $[t-t_0,t]$, i.e. the equation \eqref{EbL2diff} holds true, for all $t\ge t_0>0$.
\end{lem}
\begin{proof}
	Without loss of generality, we may assume $t=t_0$. Under the upholding assumptions, the process $ (b_t)_{t\in[0,t_0]} $  is $ \dbL^2 $-differentiable.
	By \cite[Lemma 1.3.4]{Nualart}, we know that $\zeta (X) : =- \int_{0}^{t_0} \frac{b_s}{\si^2} \mathrm{d}X_s -\int_{0}^{t_0}\frac12 |\frac{b_s}{\si}|^2 ds$ is $\dbL^2$-differentiable for any $t_0>0$, namely there exists $D \z$ such that
	\beaa
	\frac{\zeta(X+\e h) - \zeta(X)}{\eps} - \langle D\z(X), h\rangle \rightarrow 0, \q\mbox{in $\dbL^2(\dbP^{0,x}_{\si,w})$ for all $x\in\dbR^d$, as $\e\rightarrow 0$}. 
	\eeaa
	By Proposition 1.3.8 and Proposition 1.3.11 from \cite{Nualart}, we may compute $D\z$ explicitly:
	\bea\label{eq:Dzeta}
	D\z(X) = -\int_{0}^{t_0}\left(\frac{b_s}{\sigma^2} + \int_{s}^{t_0}\frac{\nabla b_r}{\sigma^2}(\mathrm{d}X_r + b_r\mathrm{d}r)\right)\mathrm{d}s.
	\eea
	Note that $\cE^b = e^\z$. Therefore, we have
	\beaa
	\cE^b(X+\e h) - \cE^b(X) = \int_0^\eps \langle \cE^b(X+s h) D\z(X+sh), h\rangle\mathrm{d}s, \q\mbox{$\dbP^{0,x}_{\si,w}$-a.s. for all $x\in\dbR^d$}. 
	\eeaa
	In order to prove \eqref{EbL2diff}, it is sufficient to prove that for all $x\in \dbR^d$
	\beaa
	\sup_{s\le 1}\dbE^{\dbP^{0,x}_{\si,w}}\Big[ \big| \langle \cE^b(X+s h) D\z(X+sh), h\rangle \big|^p \Big] <\infty,\q\mbox{for some $p>2$.}
	\eeaa
	By the form \eqref{eq:Dzeta}, we have $ \langle D\z(X+sh), h\rangle \in \cap_{q>1} \dbL^q (\dbP^{0,x}_{\si,w})$, so it is enough to show
	\bea\label{goal:EbLp}
	\dbE^{\dbP^{0,x}_{\si,w}}\Big[ \big| \cE^b(X)  \big|^p \Big] <\infty,\q\mbox{for some $p>2$.}
	\eea
	Further, note that
	\beaa
	\big| \cE^b(X)  \big|^p 
	= e^{-p \int_{0}^{t_0}\big(\si^{-2} D_m F(m_s, X_s) +  \nabla U(X_s)  \big) \mathrm{d}X_s - \frac{p}2 \int_{0}^{t_0} \frac{|b_s|^2}{\si^2} \mathrm{d}s}.
	\eeaa
	Since $D_m F$ is bounded, in order to prove \eqref{goal:EbLp}, it is enough to show that
	\beaa
	\dbE^{\dbP^{0,x}_{\si,w}}\Big[ e^{-p \int_{0}^{t_0}  \nabla U(X_s)  \mathrm{d}X_s} \Big] <\infty ,\q\mbox{for some $p>2$.}
	\eeaa
	By It\^o formula, we obtain
	\beaa
	\dbE^{\dbP^{0,x}_{\si,w}}\Big[ e^{-p \int_{0}^{t_0}  \nabla U(X_s)  \mathrm{d}X_s} \Big] 
	= \dbE^{\dbP^{0,x}_{\si,w}}\Big[ e^{-p  \Big( U(X_{t_0}) - U(x) -  \int_{0}^{t_0} \frac{\si^2}2 \Delta U(X_s)  \mathrm{d}s \Big)} \Big] <\infty,
	\eeaa
	where we use the fact that $U \ge  -C$ for some $C> 0$ and $\D U$ is bounded.
\end{proof}

\ms

\end{document}